\documentclass[english]{amsart}

\usepackage{esint}
\usepackage[svgnames]{xcolor} 
\usepackage[colorlinks,citecolor=red,pagebackref,hypertexnames=false,breaklinks]{hyperref}
\usepackage{pgf,tikz}
\usepackage{pdfsync}

\usepackage{dsfont}
\usepackage{url}
\usepackage[utf8]{inputenc}
\usepackage[T1]{fontenc}
\usepackage{lmodern}
\usepackage{babel}
\usepackage{mathtools}  
\usepackage{amssymb}
\usepackage{lipsum}
\usepackage{mathrsfs}
\usepackage{color}
\usepackage{comment}

\newtheorem{theorem}{Theorem}[section]
\newtheorem{proposition}{Proposition}[section]
\newtheorem{lemma}{Lemma}[section]

\newtheorem{corollary}{Corollary}[section]

\newtheorem{remark}{Remark}[section]

\numberwithin{equation}{section}

\title[Borg-Levinson results for Robin Laplacian with unbounded potential]{Multidimensional Borg-Levinson uniqueness and stability results for the Robin Laplacian with unbounded potential}

\author[Mourad Choulli]{Mourad Choulli}
\address{Universit\'e de Lorraine}
\email{mourad.choulli@univ-lorraine.fr}

\author[Abdelmalek Metidji]{Abdelmalek Metidji}
\address{Aix Marseille Univ, Universit\'e de Toulon, CNRS, CPT, Marseille, France}
\email{abdelmalek.metidji1@gmail.com}

\author[\'Eric Soccorsi]{\'Eric Soccorsi}
\address{Aix Marseille Univ, Universit\'e de Toulon, CNRS, CPT, Marseille, France}
\email{eric.soccorsi@univ-amu.fr}

\thanks{M. C. and \'E. S. are supported by the Agence Nationale de la Recherche (ANR) under grant ANR-17-CE40-0029 (projet MultiOnde).}

\date{\today}


\begin{document}

\begin{abstract}
%We consider the Schr\"odinger operator in a bounded domain of $\mathbb \R^n$, $n \geq 3$, endowed with Robin boundary condition. 
This article deals with the uniqueness and stability issues in the inverse problem of determining the unbounded potential of 
the Schr\"odinger operator in a bounded domain of $\mathbb{R}^n$, $n \geq 3$, endowed with Robin boundary condition,
from knowledge of its boundary spectral data.
These data are defined by the pairs formed by the eigenvalues and either partial or full Dirichlet measurement of the eigenfunctions on the boundary of the domain.
\end{abstract}

\maketitle

%%%%%%%%%%%%%%%%%%%%%%%%%%%%%%%%%%%%%%%%%%%%%%%%%%%%%%%%
%%%%%%%%%%%%%%%%%%%%%%%%%%%%%%%%%%%%%%%%%%%%%%%%%%%%%%%%
%%%%%%%%%%%                                           Preliminaries                                    %%%%%%%%%%%%%%
%%%%%%%%%%%%%%%%%%%%%%%%%%%%%%%%%%%%%%%%%%%%%%%%%%%%%%%%
%%%%%%%%%%%%%%%%%%%%%%%%%%%%%%%%%%%%%%%%%%%%%%%%%%%%%%%%

\section{Introduction}

In the present article, $\Omega$ is a $C^{1,1}$ bounded domain of $\mathbb{R}^n$, $n\ge 3$, with boundary $\Gamma$, and we equip the two spaces $H:=L^2(\Omega)$ and $V:=H^1(\Omega)$ with their usual scalar product.
Put $p:=2n/(n+2)$ and let $p^\ast:=2n/(n-2)$ be its conjugate number, in such a way that $V$ is continuously embedded in $L^{p^\ast}(\Omega)$.

To simplify notations, we denote throughout this text by $\langle \cdot,\cdot \rangle$ the duality pairing between an arbitrary Banach space and its dual.

%%%%%%%%%%%%%%%%%%%%%%%%%%%%%%%%%%%%%%%%%%%%%%%%%%%%%%%%
%%%%%%%%%%%                                 The Robin Laplacian                                 %%%%%%%%%%%%%%
%%%%%%%%%%%%%%%%%%%%%%%%%%%%%%%%%%%%%%%%%%%%%%%%%%%%%%%%

\subsection{The Robin Laplacian}
\label{sec-RL}

Let $s \in (n-1,\infty)$. For $\alpha \in L^s (\Gamma,\mathbb{R})$ and $q\in L^{n/2} (\Omega,\mathbb{R})$, we consider the sesquilinear form $\mathfrak{a} : V\times V\rightarrow \mathbb{C}$, defined by
\[
\mathfrak{a}(u,v):=\int_\Omega \nabla u\cdot \nabla \overline{v}dx+\int_\Omega qu\overline{v}dx+\mathfrak{a}_0(u,v),\quad u,v\in V,
\]
where
\[
\mathfrak{a}_0(u,v):=\int_\Gamma \alpha u\overline{v}ds(x),\quad u,v\in V.
\]
It is proved in Appendix \ref{appendixA} that $\mathfrak{a}_0$ is continuous. On the other hand, we know from \cite[Lemma 1.1]{Ch2019} that
\[
\left| \int_\Omega qu\overline{v}dx\right|\le c_\Omega \|q\|_{L^{n/2} (\Omega)}\|u\|_V\|v\|_V,
\]
where $c_\Omega >0$ is a constant depending only on $\Omega$. In consequence, $\mathfrak{a}$ is continuous. 

%where the symbol $\cdot$ stands for the usual inner product in $\mathbb{C}^n$.

Throughout the entire text, we assume that $\alpha \ge -\mathfrak{c}$ for some constant $\mathfrak{c} \in (0, \mathfrak{n}^{-2})$ almost everywhere on $\Gamma$, where $\mathfrak{n}$ denotes the norm of the (bounded) trace operator $u\in V\mapsto u_{|\Gamma}\in L^2(\Gamma)$.
Set
\[
\mathrm{Q}(\rho,\aleph):=\{q\in L^\rho(\Omega,\mathbb{R});\; \|q\|_{L^\rho(\Omega)}\le \aleph\},\quad \rho \ge n/2,\; \aleph >0.
\]
Then, arguing as in the derivation of \cite[Lemma A2]{Po}, we obtain that
\begin{equation}\label{ii1}
\|qu^2\|_{L^1(\Omega)}\le \epsilon \|u\|_V^2+C_\epsilon\|u\|_H^2,\quad q\in \mathrm{Q}(n/2,\aleph),\; u\in V,\; \epsilon >0,
\end{equation}
for some constant $C_\epsilon>0$ depending only of $n$, $\Omega$, $\aleph$ and $\epsilon$.
Further, we get by applying \eqref{ii1} with $\epsilon =\kappa:=(1-\mathfrak{c}\mathfrak{n}^2)/2$ that
\begin{equation}\label{co}
\mathfrak{a}(u,u)+ \lambda ^\ast\|u\|_H^2\ge \kappa \|u\|_V^2,\quad u\in V,
\end{equation}
where $\lambda^\ast>0$ is a constant which depends only on $n$, $\Omega$, $\mathfrak{c}$ and $\aleph$.

Then the bounded operator $A:V\rightarrow V^\ast$ defined by
\[
\langle Au,v\rangle=\mathfrak{a}(u,v),\ u,v\in V,
\]
is self-adjoint and coercive according to \eqref{co}. 
%{\color{blue} Here and in the remaining part of this article, $\langle \cdot,\cdot \rangle$ denotes the duality pairing between an arbitrary Banach space and its dual.
%}

%%%%%%%%%%%%%%%%%%%%%%%%%%%%%%%%%%%%%%%%%%%%%%%%%%%%%%%%
%%%%%%%%%%%                                               BSD                                             %%%%%%%%%%%%%%
%%%%%%%%%%%%%%%%%%%%%%%%%%%%%%%%%%%%%%%%%%%%%%%%%%%%%%%%

\subsection{Boundary spectral data} 
\label{sec-BSD}
With reference to \cite[Theorem 2.37]{Mc}, the spectrum of $A$ consists of its eigenvalues 
$\lambda_k$, $k \in \mathbb{N}:=\{1,2,\ldots \}$, arranged in non-decreasing order and repeated with the (finite) multiplicity,
\[
-\infty <\lambda_1\le \lambda_2\le \ldots \le \lambda_k\le \ldots, \quad \mbox{and\ such\ that}\quad \lim_{k \to \infty}\lambda_k \rightarrow \infty.
\]
Moreover, there exists an orthonormal basis $\{ \phi_k,\ k \in \mathbb{N} \}$ of $H$, made of eigenfunctions $\phi_k\in V$ of $A$, satisfying 
\[
\mathfrak{a}(\phi_k,v)=\lambda_k(\phi_k,v),\quad v\in V,\quad k \in \mathbb{N},
\]
where $(\cdot,\cdot)$ is the usual scalar product in $H$. For the sake of shortness, we write
\[
\psi_k:={\phi_k}_{|\Gamma},\quad k \in \mathbb{N}.
\]

Recall that for $u\in V$, we have $\Delta u\in H^{-1}(\Omega)$, the space dual to $H_0^1(\Omega)$, but that it is not guaranteed that $\Delta u$ lie in $V^\ast$ (which is strictly embedded in $H^{-1}(\Omega)$). Thus, we introduce
\[
W:=\{u\in V; \Delta u\in V^\ast\}.
\] 
Endowed with its natural norm
\[
\|u\|_{W}=\|u\|_V+\|\Delta u\|_{V^\ast},\quad u\in W,
\]
is a Banach space. 
Next, for $\varphi \in H^{1/2}(\Gamma)$, we set
\[
\dot{\varphi}:=\{v\in V;\; v_{|\Gamma}=\varphi\},
\]
and we equip the space $H^{1/2}(\Gamma)$ with its graph norm
\[
\|\varphi\|_{H^{1/2}(\Gamma)}=\min\{\|v\|_V;\; v\in \dot{\varphi}\}.
\]
Now, for $u\in W$ fixed, we put
\[
\Phi_u (v):=\langle \Delta u , v\rangle+(\nabla u,\nabla v),\quad v\in V,
\]
apply the Cauchy-Schwarz inequality, and get that
\begin{equation}\label{0.1}
|\Phi_u (v)|\le \|\Delta u\|_{V^\ast}\|v\|_V+\|u\|_V\|v\|_V \leq \|u\|_W\|v\|_V.
\end{equation}
Moreover, since $C_0^\infty (\Omega)$ is dense in $H_0^1(\Omega)$, it is easy to see that $H_0^1(\Omega)\subset \ker \Phi_u$ and consequently that $\Phi_u (v)$ depends only on $v_{|\Gamma}$. This enables us to define the normal derivative of $u$, denoted by $\partial_\nu u$, as the unique vector in $H^{-1/2}(\Gamma)$ satisfying
\[
\langle \partial_\nu u , \varphi\rangle =\Phi_u (v),\quad v \in \dot{\varphi}\hskip.2cm \mbox{is arbitrary}.
\] 
As a consequence we have 
\[
\|\partial_\nu u\|_{H^{-1/2}(\Gamma)}\le \|u\|_W,
\]
by \eqref{0.1}, and the following generalized Green formula:
\begin{equation}
\label{ggf}
\langle\Delta u , v\rangle+(\nabla u,\nabla v)=\langle \partial_\nu u , v_{|\Gamma}\rangle,\quad u\in W,\; v\in V.
\end{equation}

Pick $f\in V^\ast$ and $\mu \in \mathbb{C}$, and let $u\in V$ satisfy
\begin{equation}\label{vf}
\mathfrak{a}(u,v)+\mu(u,v)=\langle f , v\rangle ,\quad v\in V.
\end{equation}
Using that $C_0^\infty(\Omega) \subset V$, we obtain that
\[
\int_\Omega \nabla u\cdot \nabla \overline{v}dx +\int_\Omega qu\overline{v}dx+\mu\int_\Omega u\overline{v}dx=\langle f, v \rangle,\quad v\in C_0^\infty (\Omega),
\]
which yields $-\Delta u+qu+\mu u=f$ in $\mathscr{D}'(\Omega)$. Thus, bearing in mind that $qu \in V^\ast$, we have $u\in W$, and the generalized Green formula \eqref{ggf} provides
\[
\langle \partial_\nu u+\alpha u_{|\Gamma} , v_{|\Gamma}\rangle =0,\quad v\in V.
\]
Since $v\in V\mapsto v_{|\Gamma}\in H^{1/2}(\Gamma)$ is surjective, the above line reads $\partial_\nu u+\alpha u_{|\Gamma}=0$ in $H^{-1 \slash 2}(\Gamma)$, showing that \eqref{vf} is the variational formulation of the following boundary value problem (BVP):
\[
(-\Delta +q+\mu)u=f\; \mathrm{in}\; \Omega,\quad \partial_\nu u+\alpha u_{|\Gamma}=0\; \mathrm{on}\; \Gamma.
\]
As we notice from \eqref{A1} with $v=1$ that $\alpha u\in L^1(\Gamma)$ satisfies the estimate
\[
\|\alpha u\|_{L^1(\Gamma)}\le \mathbf{c}_0\|\alpha \|_{L^s(\Gamma)}\|u\|_V,
\]
for some positive constant $\mathbf{c}_0$ depending only on $s$ and $\Omega$, we see that the Robin boundary condition $\partial_\nu u+\alpha u_{|\Gamma}=0$ holds in $L^1(\Gamma)$, and hence a.e. on $\Gamma$.

Further, taking $f=0$ and $\mu=\lambda_k$ for all $k \in \mathbb{N}$, we find that $\phi_k\in W$ satisfies
\begin{equation}\label{ee}
(-\Delta +q-\lambda_k)\phi_k=0\; \mathrm{in}\; \Omega,\quad \partial_\nu \phi_k+\alpha \phi_k{_{|\Gamma}}=0\; \mathrm{on}\; \Gamma.
\end{equation}

%%%%%%%%%%%%%%%%%%%%%%%%%%%%%%%%%%%%%%%%%%%%%%%%%%%%%%%%
%%%%%%%%%%%                               Statement of the results                              %%%%%%%%%%%%%%
%%%%%%%%%%%%%%%%%%%%%%%%%%%%%%%%%%%%%%%%%%%%%%%%%%%%%%%%

\subsection{Statement of the results}
We stick to the notations of the previous sections, that is to say that we still denote by $\lambda_k$, $\phi_k$ and $\psi_k$, $k \ge \mathbb{N}$, the $k$-th eigenvalue, eigenfunction and corresponding Dirichlet trace, respectively, of the operator $A$, and we write $\tilde{\lambda}_k$ (resp., $\tilde{\phi}_k$, $\tilde{\psi}_k$) instead of $\lambda_k$ (resp., $\phi_k$, $\psi_k$) when the potential $\tilde{q}$ is substituted for $q$.
Our first result is as follows.

\begin{theorem}\label{theorem1}
Let $q$ and $\tilde{q}$ be in $L^{r}(\Omega,\mathbb{R})$, where $r=n/2$ when $n \ge 4$ and $r >n/2$ when $n=3$, and let $\ell \in \mathbb{N}$. Then, the conditions
\begin{equation}\label{thm1}
\lambda_k=\tilde{\lambda}_k\ \mbox{for\ all}\ k \ge \ell\quad \mbox{and}\quad \psi_k=\tilde{\psi}_k\ \mbox{on}\ \Gamma\ \mbox{for\ all}\ k\ge 1,
\end{equation}
yield that $q=\tilde{q}$ in $\Omega$.
\end{theorem}
The claim of Theorem \ref{theorem1} was first established for smooth bounded potentials, in the peculiar case where $\ell=1$, by Nachman, Sylvester and Uhlmann in \cite{NSU}. In the same context (of smooth bounded potentials), their result was extended to $\ell \geq 1$ through a heuristic approach in \cite{Sm}.

In view of stating our stability results, we denote by $\ell^\infty$ (resp. $\ell^2$) the Banach (resp., Hilbert) space of bounded (resp. squared summable) sequences of complex numbers $(z_k)$%(resp., sequences of complex number $(z_k)$ satisfying $\sum_{k \geq 1} |z_k|^2<\infty$)
, equipped with the norm 
\[
\|(z_k)\|_{\ell^\infty}:=\sup_{k\ge 1}|z_k|\ \left({\rm resp.,}\ \|(z_k)\|_{\ell^2}:=\left(\sum_{k\ge 1}|z_k|^2\right)^{1/2}\right),
\]  
and let 
\[
\ell^2(L^2(\Gamma)):=\left\{ (w_k) \in L^2(\Gamma)^{\mathbb N}\; \mbox{such that}\; (\|w_k\|_{L^2(\Gamma)})\in \ell^2 \right\}
\]
be endowed with its natural norm
\[
\|(w_k)\|_{\ell^2(L^2(\Gamma))}:=\|(\|w_k\|_{L^2(\Gamma)})\|_{\ell^2}.
\]
%Put $\mathrm{I}_n:=(0,1)$ if $n\ge 4$ and $\mathrm{I}_n:=(0,3/4]$ if $n=3$, in such a way that we have $n/(2\sigma)\in \left( \max(2,n/2),\infty \right)$
%whenever $\sigma \in I_n$, and set
%\[
%\mathrm{P}_\sigma(\aleph) :=\left\{ (q,\tilde{q})\in \mathrm{Q}(\max(2,n/2),\aleph)^2;\; q-\tilde{q}\in L^{n/(2\sigma)} (\Omega) \right\},\ \aleph \in (0,\infty).
%\]

\begin{theorem}\label{theorem2}
Fix $\aleph \in (0,\infty)$ and let
$(q,\tilde{q})\in \mathrm{Q}(r,\aleph)^2$, where $r=n/2$ when $n \ge 4$ and $r >n/2$ when $n = 3$, satisfy $q-\tilde{q} \in L^2(\Omega)$. Assume that $(\lambda_k-\tilde{\lambda}_k)\in \ell^\infty$ fulfills $\|(\lambda_k-\tilde{\lambda}_k)\|_{\ell^\infty}\le \aleph$ and that $(\psi_k-\tilde{\psi}_k)\in \ell^2(L^2(\Gamma))$. %For $n=3$, suppose in addition that $q-\tilde{q} \in L^2(\Omega)$. 
Then, we have
\begin{equation}\label{thm2}
\|q-\tilde{q}\|_{H^{-1}(\Omega)}\le C\left( \|(\lambda_k-\tilde{\lambda}_k)\|_{\ell^\infty}+ \|(\psi_k-\tilde{\psi}_k)\|_{\ell^2(L^2(\Gamma))} \right)^{2(1-2\beta)/(3(n+2))},
\end{equation}
where $\beta:=\max \left( 0,n(2-r)/(2r) \right)$ and $C$ is a positive constant depending only on $n$, $\Omega$, $\aleph$ and $\mathfrak{c}$.
\end{theorem}

\begin{remark}\label{remark1}
{\rm
(i) It is worth noticing that we have $\beta=0$ when $n \ge 4$, whereas $\beta \in [0,1/2)$ when $n=3$. Moreover, in the latter case we see that $\beta$ converges to $1/2$ (resp., $0$) as $r$ approaches $3/2$ (resp., $2$) from above (resp., below). \\
(ii) We have $q-\tilde{q} \in L^2(\Omega)$ for all $(q,\tilde{q}) \in \mathrm{Q}(n/2,\aleph)^2$, provided that $n \ge 4$. Nevertheless, this is no longer true when $n=3$, even if $(q,\tilde{q})$ is taken in $\mathrm{Q}(r,\aleph)^2$ with $r \in (n/2,2)$. Hence the additional requirement of Theorem \ref{theorem2} that $q-\tilde{q} \in L^2(\Omega)$ in the three-dimensional case.\\
(iii) When $q-\tilde{q} \in L^\infty(\Omega)$, we have $(\lambda_k-\tilde{\lambda}_k)\in \ell^\infty$ and $\|(\lambda_k-\tilde{\lambda}_k)\|_{\ell^\infty}\le \|q-\tilde{q}\|_{L^\infty(\Omega)}$, by the min-max principle. Thus, Theorem \ref{theorem2} remains valid by replacing the condition $\|(\lambda_k-\tilde{\lambda}_k)\|_{\ell^\infty}\le \aleph$ by the stronger assumption $\|q-\tilde{q}\|_{L^\infty(\Omega)}\le \aleph$.
\\
(iv) Assume that $|\alpha(x)|>0$ for a.e. $x \in \Gamma$. Then, the statement of Theorem \ref{theorem1} remains valid upon replacing the condition $\psi_k=\tilde{\psi}_k$ by 
$\partial_\nu \phi_k=\partial_\nu \tilde{\phi}_k$ for all  $k\ge 1$, in \eqref{thm1}. Moreover, if $1/\alpha \in L^\infty (\Gamma)$, then we may substitute 
 $\|(\psi_k-\tilde{\psi}_k)\|_{\ell^2(L^2(\Gamma))}$ by $\|(\partial_\nu \phi_k-\partial_\nu \tilde{\phi}_k)\|_{\ell^2(L^2(\Gamma))}$ on the right hand side of \eqref{thm2} in Theorem \ref{theorem2}.
}
\end{remark}

To the best of our knowledge, there is no comparable stability result available in the mathematical literature for Robin boundary conditions, even when the potentials are assumed to be bounded. Nevertheless, it should be pointed out that the variable coefficients case was recently addressed by \cite{BCKPS} in the framework of Dirichlet boundary conditions. 

Further downsizing the data needed for retrieving the unknown potential, we seek a stability inequality requesting a local Dirichlet boundary measurement of the eigenfunctions only, i.e. boundary observation of the $\psi_k$'s and $\tilde{\psi}_k$'s that is performed on a strict subset of $\Gamma$. For this purpose we assume that $\Gamma$ is connected and we consider a $C^{1,1}$-connected neighborhood  $\Omega_0$ of $\Gamma$ in $\overline{\Omega}$, a fixed nonempty open subset $\Gamma_{\ast}$ of $\Gamma$, and for all $\vartheta \in (0,\infty)$ we introduce the function $\Psi_\vartheta : [0,\infty) \to \mathbb{R}$ as
\begin{equation}
\label{def-Phi}
\Psi_\vartheta(t):=
\left\{ 
\begin{array}{cl} 
0 & \mbox{if}\ t=0
\\
|\ln t|^{-\vartheta} & \mbox{if}\  t \in (0,1/e)
\\ 
t &  \mbox{if}\  t \in [1/e,\infty).
\end{array}
\right.
\end{equation}
The corresponding local stability estimate can be stated as follows.

\begin{theorem}
\label{theorem3}
For $\aleph \in (0,\infty)$ fixed, let $(q,\tilde{q}) \in \mathrm{Q}(n,\aleph)^2$ satisfy $q=\tilde{q}$ on $\Omega_0$. 
Assume that $\alpha \in C^{0,1}(\Gamma)$, and suppose that $(\lambda_k-\tilde{\lambda}_k)\in \ell^\infty$ and that $(k^{\mathfrak{t}}(\psi_k-\tilde{\psi}_k))\in \ell^2(L^2(\Gamma))$ for some $\mathfrak{t}>4/n+1$, with
\[
\|(\lambda_k-\tilde{\lambda}_k)\|_{\ell^\infty}\le \aleph,\quad \|(k^\mathfrak{t}(\psi_k-\tilde{\psi}_k))\|_{\ell^2(L^2(\Gamma))}\le \aleph.
\]
Then there exist two constants $C>0$ and $\vartheta>0$, both of them depending only on $n$, $\Omega$, $\Omega_0$, $\Gamma_{\ast}$, $\aleph$, $\mathfrak{c}$ and $\|\alpha\|_{C^{0,1}(\Gamma)}$, such that we have:
\begin{equation}
\label{thm3}
\|q-\tilde{q}\|_{H^{-1}(\Omega)}\le C\Psi_\vartheta\left( \|(\lambda_k-\tilde{\lambda}_k)\|_{\ell^\infty}+ \|(k^{-\mathfrak{t}+2/n}(\psi_k-\tilde{\psi}_k))\|_{\ell^2(H^1(\Gamma_{\ast}))} \right).
\end{equation}
\end{theorem}

\begin{remark}\label{remark2}
{\rm
%(i) If $q$ and $\tilde{q}$ are as in the statement of Theorem \ref{theorem3} then $q$ and $\tilde{q}$ satisfy also the conditions of Theorem \ref{theorem2} with %$\sigma=1/2$.
%\\
(i) Bearing in mind that the $k$-th eigenvalue, $k \ge 1$, of the unperturbed Dirichlet Laplacian (i.e. the operator $A$ associated with $q=0$ in $\Omega$ and $\alpha=0$ on $\Gamma$) scales like $k^{2/n}$ when $k$ becomes large, see e.g. \cite[Theorem III.36 and Remark III.37]{Be}, we obtain by combining the min-max principle with \eqref{ii1}, that for all
$q\in \mathrm{Q}(n,\aleph)$,
\begin{equation}\label{waf}
C^{-1}k^{2/n}\le 1+|\lambda_k|\le Ck^{2/n},\ k\ge 1,
\end{equation}
where $C \in (1,\infty)$ is a constant depending only  on $n$, $\Omega$, $\mathfrak{c}$ and $\aleph$.
In light of Lemma \ref{lemma5.0} below, establishing the $H^2$-regularity of the eigenfunctions $\phi_k$, $k \ge 1$, of $A$, and the energy estimate \eqref{5.1}, it follows from \eqref{waf} that $(k^{-\mathfrak{t}+2/n}\psi_k)\in \ell^2(H^1(\Gamma))$. Therefore, we have $\|(k^{-\mathfrak{t}+2/n}(\psi_k-\tilde{\psi}_k))\|_{\ell^2(H^1(\Gamma_{\ast}))}<\infty$ on the right hand side of \eqref{thm3}.
\\
(ii) Assume that $\alpha\in C^1(\Gamma)$ and $1/\alpha \in L^\infty (\Gamma_\ast)$. Then, the statement of Theorem \ref{theorem3} remains valid upon replacing $\|(k^{-\mathfrak{t}+2/n}(\psi_k-\tilde{\psi}_k))\|_{\ell^2(H^1(\Gamma_\ast))}$ by $\|(k^{-\mathfrak{t}+2/n}(\partial_\nu \phi_k-\partial_\nu\tilde{\phi}_k))\|_{\ell^2(H^1(\Gamma_\ast))}$ in \eqref{thm3}. 
}
\end{remark}

%%%%%%%%%%%%%%%%%%%%%%%%%%%%%%%%%%%%%%%%%%%%%%%%%%%%%%%%
%%%%%%%%%%%                                    A short bibliography                                %%%%%%%%%%%%%%
%%%%%%%%%%%%%%%%%%%%%%%%%%%%%%%%%%%%%%%%%%%%%%%%%%%%%%%%

\subsection{A short bibliography of the existing literature}
\label{sec-literature}
The first published uniqueness result for the multidimensional Borg-Levinson problem can be found in \cite{NSU}. The breakthrough idea of the authors of this article was to relate the inverse spectral problem under analysis to the one of determining the bounded potential by the corresponding elliptic Dirichlet-to-Neumann map. This can be understood from the fact that, the Schwartz kernel of the elliptic Dirichlet-to-Neumann operator can be, at least heuristically, fully expressed in terms of the eigenvalues and the normal derivatives of the eigenfunctions. Later on, \cite{Is} proved that the result of \cite{NSU}, which assumes complete knowledge of the boundary spectral data, remains valid when finitely many of them remain unknown. 

The stability issue for multidimensional Borg-Levinson type problems was first examined in \cite{AS}. The authors proceed by
relating the spectral data to the corresponding hyperbolic Dirichlet-to-Neumann operator, which stably determines  the bounded electric potential. We refer the reader to \cite{BCY1,BCY2,BD} for alternative inverse stability results based on this approach.

In all the aforementioned results, the number of unknown spectral data is at most finite (that is to say that the data are either complete or incomplete). Nevertheless, it was proved in
\cite{CS} that asymptotic knowledge of the boundary spectral data is enough to H\"older stably retrieve the bounded potential. This result was improved in 
\cite{KKS,So} by removing all quantitative information on the eigenfunctions of the stability inequality, at the expense of an additional summability condition on their boundary measurements. The same approach was adapted to magnetic Laplacians in \cite{BCDKS}.

In all the articles cited above in this section, the unknown potential is supposed to be bounded. The unique determination of unbounded potentials by either complete or incomplete boundary spectral data is discussed in \cite{PS, Po}, whereas the stability issue for the same problem, but in the variable coefficients case, is examined in \cite{BCKPS}. As for the treatment of the inverse problem of determining the unbounded potential from asymptotic knowledge of the spectral data, we refer the reader to \cite{BKMS} for the uniqueness issue, and to \cite{KS}
for the stability issue.

All the above mentioned results were obtained for multidimensional Laplace operators endowed with Dirichlet boundary conditions, except for \cite{NSU} which proved that full knowledge of the boundary spectral data of the Robin Laplacian uniquely determines the unknown electric potential, and for \cite{BCDKS} where the case of Neumann Laplacians is examined. But, apart from the claim, based on a  heuristic approach, of \cite{Sm}, that incomplete knowledge of the spectral data of the multidimensional Robin Laplacian uniquely determines the unknown bounded potential, it seems that, even for a bounded unknown potential $q$, there is no reconstruction result of $q$ by incomplete spectral data, available in the mathematical literature for such operators. In the present article we prove not only unique identification by incomplete spectral data, but also stable determination by either full or local boundary spectral data, of the singular potential of the multidimensional Robin Laplacian.%The main achievements of this article is a 

%The case of magnetic Schr\"odinger operator was first considered in \cite{Ki} and generalized to the variable coefficients case in \cite{BCDKS}. The case of %unbounded potentials was discussed in \cite{Po}. . 

\subsection{Outline}
The remaining part of this paper is structured as follows.
In Section \ref{sec-pre} we gather several technical results  which are needed by the proof of the three main results of this article. Then we proceed with the proof of
Theorems \ref{theorem1}, \ref{theorem2} and \ref{theorem3} in Section \ref{sec-proof}.

%%%%%%%%%%%%%%%%%%%%%%%%%%%%%%%%%%%%%%%%%%%%%%%%%%%%%%%%
%%%%%%%%%%%%%%%%%%%%%%%%%%%%%%%%%%%%%%%%%%%%%%%%%%%%%%%%
%%%%%%%%%%%                                           Preliminaries                                    %%%%%%%%%%%%%%
%%%%%%%%%%%%%%%%%%%%%%%%%%%%%%%%%%%%%%%%%%%%%%%%%%%%%%%%
%%%%%%%%%%%%%%%%%%%%%%%%%%%%%%%%%%%%%%%%%%%%%%%%%%%%%%%%
\section{Preliminaries}
\label{sec-pre}

In this section we collect several preliminary results that are needed by the proof of the main results of this article. We start by noticing, upon applying \eqref{co} with $u=\phi_k$, $k\ge 1$, that 
\begin{equation}\label{lb}
\lambda_k>-\lambda^\ast,\quad k\ge 1.
\end{equation}

%Throughout the rest of this text $\rho (A)=\mathbb{C}\setminus \sigma(A)$, usually called the resolvent set of $A$.

\subsection{Resolvent estimates}

By \cite[Corollary 2.39]{Mc}, the operator $A-\lambda : V \to V^\ast$ has a bounded inverse whenever $\lambda \in \rho (A):=\mathbb{C}\setminus \sigma(A)$, the resolvent set of $A$. Furthermore, for all $f \in V^\ast$ we have
\begin{equation}\label{rf}
(A-\lambda)^{-1}f=\sum_{k\ge 1} \frac{\langle f , \phi_k \rangle}{\lambda_k-\lambda}\phi_k,
\end{equation}
where the series converges in $V$. 
%In the special case where $f \in H$, it follows readily from this and the Parseval formula that
%\begin{equation}\label{rf1}
%\|(A-\lambda)^{-1}f\|_H\le \|(1/(\lambda_k-\lambda))\|_{\ell^\infty}\|f\|_H.
%\end{equation}
For further use, we now establish that the resolvent $(A-\lambda)^{-1}$ may be regarded as a bounded operator from $H$ into the space $K:=\{u\in H;\; Au \in H\}$ endowed with the norm
\[
\|u\|_K:=\|u\|_H+\|Au\|_H,\quad u\in K.
\]

\begin{lemma}
\label{lemma1}
For all $\lambda \in \rho(A)$, the operator $(A-\lambda)^{-1}$ is bounded from $H$ into $K$.  Furthermore, we have
\begin{align}
(A-\lambda)^{-1}(A-\lambda )u=u,\quad u\in K,\label{l1.1}
\\
(A-\lambda)(A-\lambda)^{-1}f=f,\quad f\in H.\label{l1.2}
\end{align}
\end{lemma}
\begin{proof}
Put $u:=(A-\lambda)^{-1}f$ where $f\in H$ is fixed. Then, we have $(u , \phi_k )=(f,\phi_k)/(\lambda_k-\lambda)$ for all $k \ge 1$, from \eqref{rf}, whence
%Therefore, since
%\[
%Au=\sum_{k\ge 1} \lambda_k(u|\phi_k)\phi_k,
%\]
%where the series is convergent in $V^\ast$, from \cite[Theorem 2.37]{Mc}, we have
\begin{equation}
\label{es8}
Au=\sum_{k\ge1} \frac{\lambda_k}{\lambda_k-\lambda} (f,\phi_k) \phi_k,
\end{equation}
according to \cite[Theorem 2.37]{Mc}, the series being convergent in $V^\ast$.
Moreover, since
$$\sum_{k\ge 1} \frac{\lambda_k^2}{|\lambda_k-\lambda|^2} | (f,\phi_k) |^2 \le \| (\lambda_k / (\lambda_k-\lambda)) \|_{\ell^\infty}^2 \| f \|_H^2<\infty, $$
by the Parseval theorem, the right hand side on \eqref{es8} lies in $H$. Therefore, we have $Au\in H$ and 
$$\| A u \|_H \le \| (\lambda_k / (\lambda_k-\lambda)) \|_{\ell^\infty} \| f \|_H,$$ 
and consequently
$u \in K$ and 
$$\| u \|_K \le \| ((1+\lambda_k) / (\lambda_k-\lambda)) \|_{\ell^\infty} \| f \|_H.$$ 
Next, we pick $u\in K$ and set $f=(A-\lambda )u$. Then, we have $((A-\lambda)u,\phi_k)=(\lambda_k-\lambda)(u,\phi_k)$ for all $k\ge 1$, by \cite[Theorem 2.37]{Mc}, and hence
\begin{align*}
(A-\lambda)^{-1}f & =\sum_{k\ge 1}\frac{((A-\lambda)u,\phi_k)}{\lambda_k-\lambda}\phi_k \\
& =\sum_{k\ge 1} (u,\phi_k) \phi_k. 
%& =u.
\end{align*}
This establishes that $(A-\lambda)^{-1}f=u$, which yields \eqref{l1.1}. Finally, since \eqref{l1.2} follows readily from \eqref{es8}, the proof of the lemma is complete.
\end{proof}
%Armed with Lemma \ref{lemma1} we are now in position to establish the:

\begin{proposition}\label{proposition1}
Let $q\in \mathrm{Q}(n/2,\aleph)$ and let $\lambda \in \rho(A)$. Then, for all $f \in V^\ast$, the following estimate
\begin{equation}
\label{re2}
\|(A-\lambda)^{-1}f\|_V \le C \|((\lambda_k+\lambda^\ast)/(\lambda_k-\lambda))\|_{\ell^\infty} \|f\|_{V^\ast}
\end{equation}
holds with $C=\kappa^{-1/2} \| (A+\lambda^\ast)^{-1}\|_{\mathcal{B}(V^\ast,V)}$, where $\mathcal{B}(V^\ast,V)$ denotes the space of linear bounded operators from $V^\ast$ to $V$.
Moreover, in the special case where $f \in H$, we have
\begin{equation}
 \label{re1}
 \|(A-\lambda)^{-1}f\|_H\le \|(1/(\lambda_k-\lambda))\|_{\ell^\infty}\|f\|_H.
 \end{equation}
%\|(A-\lambda)^{-1}f\|_H\le |\Im \lambda |^{-1}\|f\|_H,\ f \in H
\end{proposition}

\begin{proof}
Since \eqref{re1} follows directly from \eqref{rf} and the Parseval formula, it is enough to prove \eqref{re2}. To this purpose we set $u:=(A-\lambda)^{-1} f$ and notice from the obvious identity $\Delta u = (q-\lambda) u - f \in V^\ast$ that $u \in W$. Therefore, by applying \eqref{ggf} with $v=u$, we infer from the coercivity estimate \eqref{co} that
\begin{equation}
\label{es7}
\kappa \|u\|_V^2\le \langle (A+\lambda^\ast) u , u \rangle_{V^\ast,V}.
\end{equation}
Let us assume for a while that $f \in H$. Then, with reference to \eqref{es8}, we have
$$ (A+\lambda^\ast) u = \sum_{k \ge 1} \frac{\lambda_k+\lambda^\ast}{\lambda_k-\lambda} (f,\phi_k) \phi_k, $$
where the series converges in $H$. It follows from this, \eqref{rf} and \eqref{es7} that
\begin{eqnarray}
\label{es9}
\kappa \|u\|_V^2 & \le & \sum_{k \ge 1} \frac{\lambda_k+\lambda^\ast}{|\lambda_k-\lambda|^2} |(f,\phi_k)|^2 \\
& \le & \| ( (\lambda_k+\lambda^\ast) / (\lambda_k-\lambda) ) \|_{\ell^\infty}^2
\sum_{k \ge 1} \frac{|(f,\phi_k)|^2}{\lambda_k+\lambda^\ast}. \nonumber
\end{eqnarray}
Further, taking into account that 
$$ \sum_{k \ge 1} \frac{|(f,\phi_k)|^2}{\lambda_k+\lambda^\ast} = \| (A+\lambda^\ast)^{-1} f \|_H^2,$$
according to \eqref{rf} and the Parseval formula, and then using that 
$$\| (A+\lambda^\ast)^{-1} f \|_H \le  \| (A+\lambda^\ast)^{-1}\|_{\mathcal{B}(V^\ast,V)} \| f \|_{V^\ast},$$ 
we infer from \eqref{es9} that
\begin{equation}
\label{es9b}  
\|u\|_V \leq \kappa^{-1/2}  \| (A+\lambda^\ast)^{-1}\|_{\mathcal{B}(V^\ast,V)} \| ( (\lambda_k+\lambda^\ast) / (\lambda_k-\lambda) ) \|_{\ell^\infty}\| f \|_{V^\ast}.
\end{equation}
Finally, keeping in mind that $u=(A-\lambda)^{-1} f$ and that $(A-\lambda)^{-1} \in \mathcal{B}(V^\ast,V)$, %since the usual norm in $V$ is stronger than the one in $H$, 
\eqref{re2} follows readily from \eqref{es9b} by density of $H$ in $V^\ast$.
\end{proof}

As a byproduct of Proposition \ref{proposition1}, we have the following:

\begin{corollary}\label{corollary1}
Let $q\in \mathrm{Q}(n/2,\aleph)$. Then, for all $\tau \in [1,+\infty)$ we have
\begin{equation}
\label{re4}
\|(A-(\tau+i)^2)^{-1}f\|_H\le (2\tau)^{-1}\|f\|_H,\ f \in H.
\end{equation}
Moreover, for all $\tau \ge \tau_\ast=:1+(\max(0,2-\lambda^\ast))^{1/2}$, we have
\begin{equation}
\label{re5}
\|(A-(\tau+i)^2)^{-1}f\|_V \le C (\tau+\lambda^\ast) \|f\|_{V^\ast},\ f \in V^{\ast},
\end{equation}
where $C$ is the same constant as in \eqref{re2}.
\end{corollary}
\begin{proof}
As \eqref{re4} is a straightforward consequence of \eqref{re1}, we shall only prove \eqref{re5}. To do that, we refer to \eqref{re2} and notice that
\begin{equation}
\label{es10} 
\frac{\lambda_k+\lambda^\ast}{|\lambda_k-(\tau+i)^2|}= \frac{\lambda_k+\lambda^\ast}{\left( (\lambda_k-(\tau^2-1))^2+4\tau^2 \right)^{1/2}}
\leq 2 \Theta(\lambda_k),\ k \geq 1,
\end{equation}
where we have set 
$\Theta(t):=(t+\lambda^\ast)/ (|t-(\tau^2-1)|+ 2\tau)$ for all $t \in [-\lambda^\ast,\infty)$.
Further, taking into account that $\Theta$ is a decreasing function on $[\tau^2-1,\infty)$, provided that $\tau \geq \tau_\ast$, we easily get that
$$ \sup_{t \in [-\lambda^\ast,+\infty)} \Theta(t) \le \frac{\tau^2-1+\lambda^\ast}{2 \tau} \le \frac{\tau+\lambda^\ast}{2}, $$
which along with \eqref{re2} and \eqref{es10}, yields \eqref{re5}.
\end{proof}

\begin{proposition}\label{proposition2}
Let $q \in \mathrm{Q}(n/2,\aleph)$. Then, there exists a constant $C>0$, depending only on 
$n$, $\Omega$, $\mathfrak{c}$ and $\aleph$, such that for all $\sigma \in [0,1]$ and all $f \in L^{p_\sigma}(\Omega)$, we have
\begin{equation}\label{re7}
\|(A-(\tau+i)^2)^{-1}f\|_{L^{p_\sigma^\ast}(\Omega)}\le C\tau ^{-1+2\sigma}\|f\|_{L^{p_\sigma}(\Omega)},\ \tau \in [\tau_\ast,\infty),
\end{equation}
where $p_\sigma:=2n / (n+2 \sigma)$ and $p_\sigma^\ast:=2n / (n-2 \sigma)$ is the conjugate integer to $p_\sigma$.
%where
%\[
%p_\beta:=\frac{2n}{n+2\beta},\ p_\beta^\ast:=\frac{2n}{n-2\beta}.
%\] 
\end{proposition}

\begin{proof}
In light of \eqref{re5} and the identity $p_1=p$, we have for all $f \in L^{p_1}(\Omega)$,
\begin{equation}
\label{rt1}
\| (A-(\tau+i)^2)^{-1} f \|_{L^{p_1^\ast}(\Omega)} \le C \tau \| f \|_{L^{p_1}(\Omega)},\ \tau \in [\tau_\ast,\infty),
\end{equation}
by the Sobolev embedding theorem, where $C$ is a positive constant depending only on 
$n$, $\Omega$, $\mathfrak{c}$ and $\aleph$. Further, bearing in mind that $H=L^{p_0}(\Omega)$ and that $p_0^\ast=p_0=2$, we
rewrite \eqref{re4} as
\begin{equation}
\label{rt2}
\| (A-(\tau+i)^2)^{-1} f \|_{L^{p_0^\ast}(\Omega)} \le (2 \tau)^{-1} \| f \|_{L^{p_0}(\Omega)},\ \tau \in [1,\infty),
\end{equation}
whenever $f \in L^{p_0}(\Omega)$. Therefore, since $(1-\sigma) / p_0 + \sigma / p_1 = 1 / p_\sigma$ for all $\sigma \in [0,1]$, we deduce from
\eqref{rt1}-\eqref{rt2} upon interpolating between $L^{p_0}(\Omega)$ and $L^{p_1}(\Omega)$ with the aid of the Riesz-Thorin theorem (see, e.g. \cite[Theorem IX.17]{RS2}), that
\begin{eqnarray*}
\| (A-(\tau+i)^2)^{-1} f \|_{L^{p_\sigma^\ast}(\Omega)} & \le & (C \tau)^\sigma (2\tau^{-1})^{1-\sigma} \| f \|_{L^{p_\sigma}(\Omega)} \\
& \le & 2(1+C) \tau^{-1+2\sigma} \| f \|_{L^{p_\sigma}(\Omega)},
\end{eqnarray*}
whenever $\tau \in [\tau_\ast,\infty)$. Finally, we obtain \eqref{re7} from this by renaming the constant $2(1+C)$ as $C$ in the above line.  \end{proof}

\subsection{Asymptotic spectral analysis}

Set $\mathfrak{H}:=H^2(\Omega)$ if $n\ne 4$ and put $\mathfrak{H}:=H^{2+\epsilon}(\Omega)$ for some arbitrary $\epsilon >0$, if $n=4$. We notice that $\mathfrak{H} \subset L^\infty(\Omega)$ and that the embedding is continuous, provided that $n=3$ or $n=4$, while $\mathfrak{H}$ is continuously embedded in 
$L^{2n/(n-4)}(\Omega)$ when $n>4$. The main purpose for bringing $\mathfrak{H}$ into the analysis here is the following useful property: $fu\in H$
whenever $f\in L^{\max(2,n/2)}(\Omega)$ and $u\in \mathfrak{H}$.

Next we introduce the subspace
\[
\mathfrak{h}:=\{ g=\partial_\nu G +\alpha G_{|\Gamma};\; G\in \mathfrak{H}\}
\]
of $L^2(\Gamma)$, equipped with its natural quotient norm
\[
\|g\|_{\mathfrak{h}}:=\min\{ \|G\|_{\mathfrak{H}};\; G\in \dot{g}\},\quad g\in \mathfrak{h},
\]
where
\[
\dot{g}:=\{ G\in \mathfrak{H};\; \partial_\nu G+\alpha G_{|\Gamma}=g\},\quad g\in \mathfrak{h},
\]
and we consider the non homogenous BVP:
\begin{equation}\label{bvp1}
(-\Delta +q-\lambda )u=0\; \mathrm{in}\; \Omega ,\quad \partial_\nu u+\alpha u_{|\Gamma} =g\; \mathrm{on}\; \Gamma .
\end{equation}

We first examine the well-posedness of \eqref{bvp1}.

\begin{lemma}\label{lemma2}
Let $\lambda\in \rho(A)$ and let $g\in \mathfrak{h}$. Then, the function
\begin{equation}\label{sol1}
u_\lambda (g):=(A-\lambda )^{-1}(\Delta -q+\lambda)G+G
\end{equation}
is independent of $G\in \dot{g}$. Moreover, $u_\lambda (g)\in  K\cap W$ is the unique solution to \eqref{bvp1} and is expressed as
\begin{equation}\label{rep1}
u_\lambda (g)=\sum_{k\ge 1}\frac{(g,\psi_k)}{\lambda_k-\lambda}\phi_k
\end{equation} 
in $H$. Here $(\cdot,\cdot)$ stands for the usual scalar product in $L^2(\Gamma)$, not to be mistaken with the scalar 
product in $H$ which is denoted by the same symbol.
\end{lemma}

\begin{proof}
Since $G\in \mathfrak{H}$, it is clear that $(\Delta -q+\lambda)G\in H$. Thus, the right hand side of \eqref{sol1} lies in $W$ and it is obviously a solution to  the BVP \eqref{bvp1}. Moreover, $\lambda$ being taken in the resolvent set of $A$, this solution is unique.

Further, for all $G_1$ and $G_2$ in $\dot{g}$, it is easy to check that $\partial_\nu (G_1-G_2) + \alpha (G_1-G_2)=0$ on $\Gamma$ and that $(A-\lambda )^{-1}(\Delta -q+\lambda)(G_1-G_2)=-(G_1-G_2)$ in $\Omega$. Therefore,
the function $u_\lambda (g)$ given by \eqref{sol1}, is independent of $G\in \dot{g}$. 

We turn now to showing \eqref{rep1}. To do that we apply the generalized Green formula \eqref{ggf} with $u=u_\lambda(g)$ and $v=\phi_k$, $k \geq 1$. We obtain
\[
\langle \Delta u_\lambda(g) , \phi_k\rangle+(\nabla u_\lambda(g)|\nabla \phi_k)=\langle \partial_\nu  u_\lambda(g) , \psi_k\rangle,
\]
which may be equivalently rewritten as
\begin{equation}\label{2.1}
((q-\lambda) u_\lambda(g),\phi_k)+(\nabla u_\lambda(g),\nabla \phi_k)=\langle g-\alpha u_\lambda(g)_{|\Gamma} ,\psi_k\rangle.
\end{equation}
Doing the same with $u=\phi_k$ and $v=u_\lambda(g)$, and taking the conjugate of both sides of the obtained equality, we find that
$$( u_\lambda(g), (q-\lambda_k) \phi_k)+( \nabla u_\lambda(g), \nabla \phi_k)=-\langle   u_\lambda(g)_{|\Gamma} , \alpha \psi_k \rangle.
$$
Bearing in mind that $q$ and $\alpha$ are real-valued, and that $\lambda_k \in \mathbb{R}$, this entails that
\begin{equation}\label{2.2}
((q-\lambda_k)u_\lambda(g),\phi_k)+(\nabla u_\lambda(g),\nabla \phi_k)=-\langle \alpha u_\lambda(g)_{|\Gamma} , \psi_k\rangle.
\end{equation}
Now, taking the difference of \eqref{2.1} with \eqref{2.2}, we end up getting that
\[
(\lambda_k-\lambda )(u_\lambda(g),\phi_k)=\langle g , \psi_k \rangle=(g,\psi_k).
\]
This and the basic identity
\[
u_\lambda(g)=\sum_{k\ge 1}(u,\phi_k)\phi_k
\]
yield \eqref{rep1}.
\end{proof}

The series on the right hand side of \eqref{rep1} converges only in $H$ and thus we cannot deduce an expression of the trace $u_\lambda(g)_{| \Gamma}$ in terms of $\lambda_k$ and $\psi_k$, $k \geq 1$, directly from \eqref{rep1}. To circumvent this difficulty we 
establish the following lemma:

\begin{lemma}\label{lemma4}
Let $g\in \mathfrak{h}$. Then, for all $\lambda$ and $\mu$ in $\rho(A)$, we have
\begin{equation}\label{s2}
u_\lambda(g){_{|\Gamma}} -u_\mu(g){_{|\Gamma}}=(\lambda-\mu)\sum_{k\ge 1}\frac{(g,\psi_k)}{(\lambda_k-\lambda)(\lambda_k-\mu)}\psi_k,
\end{equation}
and the series converges in $H^{1/2}(\Gamma)$.
\end{lemma}

\begin{proof} 
Notice that
\[
(-\Delta +q-\lambda )(u_\lambda -u_\mu)=(\lambda-\mu)u_\mu
\]
in $\Omega$ and that $\partial_\nu(u_\lambda -u_\mu)+\alpha(u_\lambda -u_\mu)_{|\Gamma}=0$ on $\Gamma$, where, for shortness sake, we write $u_\lambda =u_\lambda (g)$ and $u_\mu=u_\mu(g)$.
Thus, we have
\[
u_\lambda -u_\mu=(\lambda-\mu)(A-\lambda)^{-1}u_\mu =(\lambda-\mu)\sum_{k\ge 1}\frac{(u_\mu,\phi_k)}{\lambda_k-\lambda}\phi_k.
\]
On the other hand, since 
\[
(u_\mu,\phi_k)=\frac{(g,\psi_k)}{\lambda_k-\mu},\ k \geq 1,
\]
from \eqref{rep1}, we obtain that
\begin{equation}\label{s1}
u_\lambda -u_\mu=(\lambda-\mu)\sum_{k\ge 1}\frac{(g,\psi_k)}{(\lambda_k-\lambda)(\lambda_k-\mu)}\phi_k,
\end{equation}
where the series converges in $K$. As a consequence we have
\begin{align*}
 (A-\lambda)\sum_{k\ge 1}\frac{(g,\psi_k)}{(\lambda_k-\lambda)(\lambda_k-\mu)}\phi_k= & \sum_{k\ge 1}\frac{(g,\psi_k)}{(\lambda_k-\lambda)(\lambda_k-\mu)}(A-\lambda)\phi_k \\
& =\sum_{k\ge 1}\frac{(g,\psi_k)}{\lambda_k-\mu}\phi_k,
\end{align*}
the series being convergent in $H$,  whence 
\[
\sum_{k\ge 1}\frac{(g,\psi_k)}{(\lambda_k-\lambda)(\lambda_k-\mu)}\phi_k=(A-\lambda)^{-1}\sum_{k\ge 1}\frac{(g,\psi_k)}{\lambda_k-\mu}\phi_k,
\] 
according to \eqref{l1.1}.

It follows from this and \eqref{s1} that
\[
u_\lambda -u_\mu=(\lambda-\mu) (A-\lambda)^{-1}\sum_{k\ge 1}\frac{(g,\psi_k)}{\lambda_k-\mu}\phi_k,
\]
where the series on the right hand side of \eqref{s1} converges in $V$. As a consequence we have
\begin{equation}%\label{s2}
u_\lambda{_{|\Gamma}} -u_\mu{_{|\Gamma}}=(\lambda-\mu)\sum_{k\ge 1}\frac{(g,\psi_k)}{(\lambda_k-\lambda)(\lambda_k-\mu)}\psi_k,
\end{equation}
the series being convergent in $H^{1/2}(\Gamma)$.
\end{proof}

Next, we establish the following {\it a priori} estimate for the solution to \eqref{bvp1}.

\begin{lemma}\label{lemma3}
Let $q\in \mathrm{Q}(n/2,\aleph)$. Then, there exist two constants $\lambda_+\ge \lambda^\ast$ and $C>0$, depending only on $n$, $\Omega$, $\aleph$ and $\mathfrak{c}$, such that for all $\lambda \in (-\infty,-\lambda_+]$ and all $g\in \mathfrak{h}$, the solution $u_\lambda (g)$ to \eqref{bvp1} satisfies the estimate
\begin{equation}\label{lim1}
|\lambda|^{1/2} \|u_\lambda(g)\|_H+\|u_\lambda (g)\|_V\le C\|g\|_{L^2(\Gamma)}.
\end{equation}
\end{lemma}

\begin{proof}
Fix $\lambda \in \rho(A)\cap (-\infty ,0)$. We apply the generalized Green formula \eqref{ggf}
with $u=v:=u_\lambda$, where we write $u_\lambda$ instead of $u_\lambda(g)$. We get that
\begin{equation}\label{ae1}
|\lambda| \|u_\lambda\|_H^2+\|\nabla u_\lambda\|_H^2 \leq \|qu_\lambda ^2\|_{L^1(\Omega)}-(\alpha u_\lambda,u_\lambda)+(g,u_\lambda).
\end{equation}
Next, $\epsilon$ being fixed in $(0,+\infty)$, we combine \eqref{ii1} with \eqref{ae1} and obtain
\begin{equation}\label{ae2}
|\lambda |\|u_\lambda\|_H^2+\|\nabla u_\lambda\|_H^2\le \epsilon\|u_\lambda\|_V^2+C_\epsilon \|u_\lambda\|_H^2+\mathfrak{c}\mathfrak{n} ^2\|u_\lambda\|_V^2+\mathfrak{n} \|g\|_{L^2(\Gamma)}\|u_\lambda\|_V,
\end{equation}
where $C_\epsilon$ is a positive constant depending only on $n$, $\Omega$, $\aleph$ and $\epsilon$.
Taking $\epsilon = \kappa =(1- \mathfrak{c}\mathfrak{n}^2)/2$ in \eqref{ae2} then yields
\[
(|\lambda |-1-C_\kappa)\|u_\lambda\|_H^2+\kappa\|u_\lambda\|_V^2\le \mathfrak{n} \|g\|_{L^2(\Gamma)}\|u_\lambda\|_V.
\]
As a consequence we have 
$$ | \lambda | \| u_\lambda \|_H^2 + \| u_\lambda \|_V^2 \leq \frac{2\mathfrak{n}^2}{\kappa^2} \|g\|_{L^2(\Gamma)}^2, $$
whenever $|\lambda| \geq (1+C_\kappa) \slash (1 - \kappa \slash 4)$, and 
\eqref{lim1} follows readily from this.
\end{proof}

%Interpolating between $H$ and $V$, we infer from Lemma \ref{lemma3} that
%\begin{equation}\label{lim1.1}
%\|u_\lambda(g)\|_{H^s(\Omega)}\le C_s|\lambda|^{(s-1)/2}\|g\|_{L^2(\Gamma)},\quad -\lambda\ge \lambda_+,\; 0\le s\le 1,
%\end{equation}
%for some constant $C_s>0$ which depends only on $n$, $\Omega$, $\aleph$, $\mathfrak{c}$ and $s$.

Armed with Lemma \ref{lemma3} we can examine the dependence of (the trace of) the solution to the BVP \eqref{bvp1} with respect to $q$. 
More precisely, we shall establish that the influence of the potential on $u_\lambda(g)$ is, in some sense, dimmed as the spectral parameter $\lambda$ goes to $-\infty$. 
\begin{lemma}\label{lemma5.1}
Let $q$ and $\tilde{q}$ be in $\mathrm{Q}(n/2,\aleph)$. 
%Assume moreover that $q-\tilde{q}\in L^{n/s}(\Omega)$ for some $s \in (0,1)$. 
Then, for all
$g\in \mathfrak{h}$, we have
\begin{equation}\label{lim2}
\lim_{\lambda=\Re \lambda \rightarrow -\infty}\|u_\lambda(g)_{|\Gamma}-\tilde{u}_\lambda(g)_{|\Gamma}\|_{H^{1/2}(\Gamma)}=0.
\end{equation}
\end{lemma}

\begin{proof}
Let $\lambda \in (-\infty,-\lambda_+]$, where $\lambda_+$ is the same as in Lemma \ref{lemma3}. We use the same notation as in the proof of Lemma \ref{lemma3} and write $u_\lambda$ (resp., $\tilde{u}_\lambda$) instead of $u_\lambda(g)$ (resp., $\tilde{u}_\lambda(g)$). Since
\[
(-\Delta +q-\lambda )(u_\lambda - \tilde{u}_\lambda)=(\tilde{q}-q)\tilde{u}_\lambda\quad \mathrm{in}\; \Omega
\]
and 
\[
\partial_\nu(u_\lambda - \tilde{u}_\lambda)+\alpha (u_\lambda - \tilde{u}_\lambda)_{|\Gamma}=0\quad \mbox{on}\; \Gamma, 
\]
we have
$$u_\lambda - \tilde{u}_\lambda=(A-\lambda)^{-1} ((\tilde{q}-q)\tilde{u}_\lambda), $$
whence
\begin{equation}
\label{a1}
\| u_\lambda - \tilde{u}_\lambda \|_V \le C \|((\lambda_k+\lambda^\ast)/(\lambda_k-\lambda))\|_{\ell^\infty} \| (\tilde{q}-q)\tilde{u}_\lambda\|_{V^\ast},
\end{equation}
by \eqref{re2}, where $C$ is a positive constant which is independent of $\lambda$. 

We are left with the task of estimating $\| (\tilde{q}-q)\tilde{u}_\lambda\|_{V^\ast}$. For this purpose, we notice from $\tilde{q}-q \in L^{n/2}(\Omega)$ and from $\tilde{u}_\lambda \in L^{p^\ast}(\Omega)$ that $(\tilde{q}-q)\tilde{u}_\lambda \in L^p(\Omega)$. Thus, bearing in mind that the embedding $V \subset L^{p^\ast}(\Omega)$ is continuous, we infer from H\"older's inequality that
\begin{eqnarray*}
\| (\tilde{q}-q)\tilde{u}_\lambda\|_{V^\ast} %& \le & \| (\tilde{q}-q)\tilde{u}_\lambda\|_{L^p(\Omega)} \\
& \le & \| \tilde{q}-q \|_{L^{n/2}(\Omega)} \| \tilde{u}_\lambda \|_{L^{p^\ast}(\Omega)} \\
& \le & 2 \aleph \| \tilde{u}_\lambda \|_{V}.
\end{eqnarray*}
In light of \eqref{lim1}, this entails that 
$$\| (\tilde{q}-q)\tilde{u}_\lambda\|_{V^\ast} \leq C \|g\|_{L^2(\Gamma)},$$
for some constant $C$ depending only on $n$, $\Omega$, $\aleph$ and $\mathfrak{c}$. From this, \eqref{a1} and the continuity of the trace operator $w\in V\mapsto w_{|\Gamma}\in H^{1/2}(\Gamma)$, we obtain that
\[
\|{u_\lambda}_{|\Gamma}-{\tilde{u}_\lambda}{_{|\Gamma}} \|_{H^{1/2}(\Gamma)}\le C \|((\lambda_k+\lambda^\ast)/(\lambda_k-\lambda))\|_{\ell^\infty} \|g\|_{L^2(\Gamma)},
\]
where $C$ is independent of $\lambda$. Now \eqref{lim2} follows immediately from this upon sending $\lambda$ to $-\infty$ on both sides of the above inequality.
\end{proof}

\subsection{$H^2$-regularity of the eigenfunctions} 
For all $q \in L^{n/2}(\Omega)$,
we have $\phi_k \in V$, $k \ge 1$, but it is no guaranteed in general that $\phi_k \in H^2(\Omega)$. 
Nevertheless, we shall establish that the regularity of the eigenfunctions of $A$ can be upgraded to $H^2$, provided that the potential $q$ is taken in  $L^n(\Omega)$.

\begin{lemma}\label{lemma5.0}
Let $q\in \mathrm{Q}(n,\aleph)$ and assume that $\alpha\in C^{0,1}(\Gamma)$. Then, for all $k\in \mathbb{N}$, we have $\phi_k\in H^2(\Omega)$ and the estimate
\begin{equation}\label{5.1}
\|\phi_k\|_{H^2(\Omega)}\le C(1+|\lambda_k|),
\end{equation}
where $C$ is a positive constant depending on $n$, $\Omega$ and $\aleph$ and $\|\alpha\|_{C^{0,1}(\Gamma)}$.
\end{lemma}

\begin{proof} 
Let us start by noticing from \eqref{co} that
\begin{equation}\label{5.0}
\|\phi_k\|_V\le \kappa^{-1/2}(\lambda_k+\lambda ^\ast)^{1/2},\ k \geq 1.
\end{equation}
On the other hand we have $q\phi_k\in H$ for all $k\in \mathbb{N}$, and the estimate
 \begin{equation}\label{5.0.1}
 \|q\phi_k\|_H\le \|q\|_{L^n(\Omega)}\|\phi_k\|_{L^{p^\ast}(\Omega)}\le C_0 \|\phi_k\|_V,
 \end{equation}
where $C_0$ is a positive constant depending only on $n$, $\Omega$, $\mathfrak{c}$ and $\aleph$.
 
Next, bearing in mind that $\alpha \phi_k{_{|\Gamma}} \in H^{1/2}(\Gamma)$, we pick
$\phi_k^0\in H^2(\Omega)$ such that $\partial_\nu \phi_k^0=\alpha \phi_k{_{|\Gamma}}$. Evidently, we have
 \[
 -\Delta (\phi_k+\phi_k^0)=(\lambda_k-q)\phi_k-\Delta\phi_k^0\; \mbox{in}\; \Omega \quad \mbox{and} \quad \partial_\nu(\phi_k+\phi_k^0)=0\; \mbox{on}\; \Gamma.
 \]
Since $(\lambda_k-q)\phi_k-\Delta\phi_k^0\in H$, \cite[Theorem 3.17]{Tr} then yields that $\phi_k+\phi_k^0\in H^2(\Omega)$. As a consequence we have $\phi_k=(\phi_k+\phi_k^0) -\phi_k^0 \in H^2(\Omega)$ and 
\[
\|\phi_k\|_{H^2(\Omega)}\le C_1(\|(\lambda_k-q)\phi_k\|_H+\|\phi_k\|_V)
\]
for some constant $C_1>0$ which depends only on $n$, $\Omega$ and $\|\alpha\|_{C^{0,1}(\Gamma)}$, by \cite[Lemma 3.181]{Tr} (see also \cite[Theorem 2.3.3.6]{Gr}). Putting this together with \eqref{5.0}-\eqref{5.0.1}, we obtain \eqref{5.1}.
\end{proof}

%%%%%%%%%%%%%%%%%%%%%%%%%%%%%%%%%%%%%%%%%%%%%%%%%%%%%%%%
%%%%%%%%%%%%%%%%%%%%%%%%%%%%%%%%%%%%%%%%%%%%%%%%%%%%%%%%
%%%%%%%%%%%%%%%%%%%%%%%%%%%%%%%%%%%%%%%%%%%%%%%%%%%%%%%%
%%%%%%%%%%%                      Proof of Theorems  1.1, 1.2 & 1.3                        %%%%%%%%%%%%%%
%%%%%%%%%%%%%%%%%%%%%%%%%%%%%%%%%%%%%%%%%%%%%%%%%%%%%%%%
%%%%%%%%%%%%%%%%%%%%%%%%%%%%%%%%%%%%%%%%%%%%%%%%%%%%%%%%
%%%%%%%%%%%%%%%%%%%%%%%%%%%%%%%%%%%%%%%%%%%%%%%%%%%%%%%%

\section{Proof of Theorems \ref{theorem1}, \ref{theorem2} and \ref{theorem3}}
\label{sec-proof}

%%%%%%%%%%%%%%%%%%%%%%%%%%%%%%%%%%%%%%%%%%%%%%%%%%%%%%%%
%%%%%%%%%%%                                Proof of Theorem 1.1                                 %%%%%%%%%%%%%%
%%%%%%%%%%%%%%%%%%%%%%%%%%%%%%%%%%%%%%%%%%%%%%%%%%%%%%%%

\subsection{Proof of Theorem \ref{theorem1}}
\label{sec-prthm1}
%Let $q,\tilde{q}\in L^{n/2}(\Omega,\mathbb{R})$ and assume that, for some $\ell \ge 1$,
%\[
%\lambda_k=\tilde{\lambda}_k,\quad k\ge \ell \quad \mathrm{and}\quad \psi_k=\tilde{\psi}_k,\quad k\ge 1.
%\]
We use the same notations as in the previous sections. Namely, we denote by $\tilde{A}$ is the operator generated in $H$ by $\mathfrak{a}$ where $\tilde{q}$ is substituted for $q$, and we write $u_\lambda$ (resp., $\tilde{u}_\lambda$) instead of $u_\lambda (g)$ (resp., $\tilde{u}_\lambda(g)$). 
Let $\lambda \in \mathbb{C}\setminus \mathbb{R}$ and pick $\mu$ in $\rho(A)\cap \rho(\tilde{A})$. Depending on whether $\ell=1$ or $\ell \ge 2$, we have either
\[
{u_\lambda}_{|\Gamma} -{u_\mu}_{|\Gamma}={\tilde{u}_\lambda}{_{|\Gamma}} - {\tilde{u}_\mu}{_{|\Gamma}}
\]
or
\begin{eqnarray*}
& & {u_\lambda}_{|\Gamma} -{u_\mu}_{|\Gamma}-(\lambda-\mu) \sum_{k= 1}^{\ell-1}\frac{(g,\psi_k)}{(\lambda_k-\lambda)(\lambda_k-\mu)}\psi_k
\\
&=& {\tilde{u}_\lambda}{_{|\Gamma}} -{\tilde{u}_\mu}{_{|\Gamma}}-(\lambda-\mu)\sum_{k= 1}^{\ell-1}\frac{(g,\psi_k)}{(\tilde{\lambda}_k-\lambda)(\tilde{\lambda}_k-\mu)}\psi_k,
\end{eqnarray*}
by virtue of \eqref{s2}. Sending $\Re \mu$ to $-\infty$ in these two identities, where $\Re \mu$ denotes the real part of $\mu$, we get with the help of \eqref{lim2} that
\begin{equation}\label{RtoD}
{u_\lambda}_{|\Gamma}- {\tilde{u}_\lambda}{_{|\Gamma}}=R_\lambda^\ell,
\end{equation}
where 
\[
R_\lambda^\ell=R_\lambda^\ell (g) := \left\{ \begin{array}{ll} 0 & \mbox{if}\ \ell=1 \\ \sum_{k= 1}^{\ell-1}\frac{(\tilde{\lambda}_k -\lambda_k)(g,\psi_k)}{(\lambda_k-\lambda)(\tilde{\lambda}_k-\lambda)}\psi_k & \mbox{if}\ \ell \ge 2. \end{array} \right. 
\]
Notice for further use that there exists $\lambda_\ast>0$ such that the estimate
\begin{equation}\label{Re}
| ( R_\lambda ^\ell ,h ) |\le \frac{C_\ell}{|\lambda|^2}\|g\|_{L^2(\Gamma)}\|h\|_{L^2(\Gamma)},\quad |\lambda |\ge \lambda_\ast,\quad g,h\in \mathfrak{h}, %\quad \ell \in \mathbb{N},
\end{equation}
holds for some constant $C_\ell=C_\ell(q,\tilde{q})$ which is independent of $\lambda$.

Let us now consider two functions $G \in \mathfrak{H}$  and $H \in \mathfrak{H}$, that will be made precise below, and put $u:=(A-\lambda)^{-1}(\Delta -q+\lambda)G+G$, $g:=\partial_\nu G+\alpha G_{|\Gamma}$ and $h:=\partial_\nu H +\alpha H_{|\Gamma}$. Then, bearing in mind that $\partial_\nu u + u_{|  \Gamma}=g$, the Green formula yields that
\begin{equation}
\label{eq-G}
\int_\Gamma u\overline{h}ds(x)=\int_\Gamma g\overline{H}ds(x)+\int_\Omega (u\Delta\overline{H}-\Delta u\overline{H})dx.
\end{equation}
Further, taking into account that $\Delta u=(q-\lambda )u$ in $\Omega$, we see that
\begin{eqnarray*}
u\Delta\overline{H}-\Delta u\overline{H} & =& u (\Delta -q+\lambda )\overline{H} \\
& = & \left( (A-\lambda)^{-1}(\Delta -q+\lambda )G+G \right) (\Delta -q+\lambda )\overline{H}.
\end{eqnarray*}
Thus, assuming that $(\Delta +\lambda)G=(\Delta +\lambda)H=0$, the above identity reduces to
\[
u\Delta\overline{H}-\Delta u\overline{H}= -\left( -(A-\lambda)^{-1}qG+G \right) q\overline{H},
\]
and \eqref{eq-G} then reads
\begin{equation}\label{id1}
\int_\Gamma u\overline{h}ds(x)=\int_\Gamma g\overline{H}ds(x)-\int_\Omega \left( -(A-\lambda)^{-1}qG+G \right) q\overline{H} dx.
\end{equation}

This being said, we set $\lambda_\tau:=(\tau+i)^2$ for some fixed $\tau \in [1,+\infty)$, pick two vectors $\omega$ and $\theta$ in $\mathbb{S}^{n-1}$, and we consider the special case where
\[
G(x)=\mathfrak{e}_{\lambda_\tau,\omega}(x):=e^{i\sqrt{\lambda_\tau}\omega \cdot x},\quad \overline{H}(x)=\mathfrak{e}_{\lambda_\tau,-\theta}(x):= e^{-i\sqrt{\lambda_\tau}\theta \cdot x}. 
\]
Next, we put
$$
S(\lambda_\tau,\omega ,\theta) :=\int_\Gamma u_\lambda (g)\overline{h}ds(x),\ \quad \tilde{S}(\lambda_\tau,\omega ,\theta) :=\int_\Gamma \tilde{u}_\lambda (g)\overline{h}ds(x),
$$
in such a way that
\begin{equation}
\label{eq-S}
S(\lambda_\tau,\omega ,\theta)-\tilde{S}(\lambda,\omega ,\theta)=\langle R_{\lambda_\tau} ^\ell(g) , h\rangle.
\end{equation}
Then, taking into account that
\[
g(x)=(i\sqrt{\lambda_\tau}\omega \cdot \nu+\alpha)e^{i\sqrt{\lambda_\tau}\omega \cdot x},\quad \overline{h}(x)= (-i\sqrt{\lambda_\tau}\theta \cdot \nu+\alpha)e^{-i\sqrt{\lambda_\tau}\theta \cdot x},
\]
we have $\|g\|_{L^2(\Gamma)}\|h\|_{L^2(\Gamma)}\le C\tau^2$ for some positive constant $C$ which is independent of $\omega$, $\theta$ and $\tau$, and we infer from \eqref{Re} and \eqref{eq-S} that
\begin{equation}
\label{S1}
\lim_{\tau \rightarrow \infty}\sup_{\omega,\theta \in \mathbb{S}^{n-1}} \left( S(\lambda_\tau,\omega ,\theta)-\tilde{S}(\lambda_\tau,\omega ,\theta) \right)=0.
\end{equation}
On the other hand, \eqref{id1} reads
\begin{equation}
\label{S1b}
S(\lambda_\tau,\omega ,\theta)  =S_0(\lambda_\tau,\omega ,\theta)  
+\int_\Gamma(i\sqrt{\lambda_\tau}\omega \cdot\nu +\alpha)e^{-i\sqrt{\lambda_\tau}(\theta-\omega)\cdot x}ds(x),
\end{equation}
where
\begin{equation}
\label{id3}
S_0(\lambda_\tau,\omega ,\theta):=\int_\Omega (A-\lambda_\tau)^{-1}(q\mathfrak{e}_{\lambda_\tau,\omega})q\mathfrak{e}_{\lambda_\tau,-\theta}dx-\int_\Omega qe^{-i\sqrt{\lambda_\tau}(\theta-\omega)\cdot x}dx.
\end{equation}
Now, we fix $\xi$ in $\mathbb{R}^n$, pick $\eta \in \mathbb{S}^{n-1}$ such that $\xi \cdot \eta =0$, and for all $\tau \in \left( |\xi|/2,+\infty \right)$ we set 
\begin{equation}
\label{es4}
\omega_\tau :=\left(1-|\xi|^2/(4\tau ^2)\right)^{1/2}\eta -\xi/(2\tau),\quad \theta_\tau :=\left(1-|\xi|^2/(4\tau ^2)\right)^{1/2}\eta +\xi/(2\tau)
\end{equation}
in such a way that
\begin{equation}
\label{id3b}
\lim_{\tau \rightarrow +\infty} \sqrt{\lambda_\tau}(\theta_\tau-\omega_\tau) = \xi.
\end{equation}
Evidently, we have
\begin{equation}
\label{gs0}
\|\mathfrak{e}_{\lambda_\tau,\omega_\tau}\|_{L^\infty (\Omega)}\le \|e^{|x|} \|_{L^\infty (\Omega)},\quad \|\mathfrak{e}_{\lambda_\tau,-\theta_\tau}\|_{L^\infty (\Omega)}\le \|e^{|x|} \|_{L^\infty (\Omega)}.
\end{equation}
Next, with reference to the notations $\beta=\max \left( 0,n(2-r)/(2r) \right)$ and $p_\sigma=2n / (n+2\sigma)$, $\sigma \in [0,1]$, of Theorem \ref{theorem2} and Proposition \ref{proposition2}, respectively,
we see that $\beta=0$ and hence that $p_\beta=p_0=2$, when $n \ge 4$, whereas $p_\beta=r \in (3/2,2)$, when $n=3$. Thus, we have $p_\beta\le r$ whenever $n \ge 3$, and consequently $q \in L^{p_\beta}(\Omega)$. It follows from this and \eqref{gs0} that
$q\mathfrak{e}_{\lambda_\tau,\omega_\tau}$ and $q\mathfrak{e}_{\lambda_\tau,-\theta_\tau}$ lie in $L^{p_\beta}(\Omega)$ and satisfy the estimate
\begin{equation}
\label{gs1}
\|q\mathfrak{e}_{\lambda_\tau,\omega_\tau}\|_{L^{p_\beta}(\Omega)}+\|q\mathfrak{e}_{\lambda_\tau,-\theta_\tau} \|_{L^{p_\beta}(\Omega)}\le C \| q \|_{L^r(\Omega)},\quad \tau \in (|\xi|/2,\infty),
\end{equation}
for some positive constant $C=C(n,\Omega)$ depending only on $n$ and $\Omega$. Moreover, for all $\tau \ge \max(|\xi|/2,\tau_\ast)$, we have
\begin{eqnarray}
\label{gm0}
& & \left| \int_\Omega (A-\lambda_\tau)^{-1}(q\mathfrak{e}_{\lambda_\tau,\omega_\tau})q\mathfrak{e}_{\lambda_\tau,-\theta_\tau}dx \right| \\
& \leq & \| (A-\lambda_\tau)^{-1}(q\mathfrak{e}_{\lambda_\tau,\omega_\tau}) \|_{L^{p_\beta^\ast}(\Omega)} \|q\mathfrak{e}_{\lambda_\tau,-\theta_\tau} \|_{L^{p_\beta}(\Omega)} \nonumber\\
& \le & C \tau^{-1+2\beta} \|q\mathfrak{e}_{\lambda_\tau,\omega_\tau}\|_{L^{p_\beta}(\Omega)} \|q\mathfrak{e}_{\lambda_\tau,-\theta_\tau} \|_{L^{p_\beta}(\Omega)},\ \nonumber
\end{eqnarray}
by \eqref{re7}, where $C>0$ is independent of $\tau$. Since $\beta \in [0,1/2)$ from its very definition, we infer from
\eqref{gs1}-\eqref{gm0} that
\begin{equation}
\label{gs2}
\lim_{\tau \rightarrow \infty}\left| \int_\Omega (A-\lambda_\tau)^{-1}(q\mathfrak{e}_{\lambda_\tau,\omega_\tau})q\mathfrak{e}_{\lambda_\tau,-\theta_\tau}dx \right|=0,
\end{equation}
which together with \eqref{id3}-\eqref{id3b} yields that
\[
\lim_{\tau \rightarrow \infty} S_0(\lambda_\tau ,\omega_\tau ,\theta_\tau) =-\int_\Omega qe^{-i\xi \cdot x},\quad \xi \in \mathbb{R}^n.
\]
From this and the identity
\[
\lim_{\tau \rightarrow \infty}\left( S_0(\lambda_\tau ,\omega_\tau ,\theta_\tau)-\tilde{S}_0(\lambda_\tau ,\omega_\tau ,\theta_\tau) \right)=\lim_{\tau \rightarrow \infty} \left( S(\lambda_\tau,\omega_\tau ,\theta_\tau)-\tilde{S}(\lambda_\tau,\omega_\tau ,\theta_\tau) \right)=0,
\]
arising from \eqref{S1}-\eqref{S1b}, it then follows that
\[
\int_\Omega (q-\tilde{q})e^{-i\xi \cdot x}dx=0,\quad \xi \in \mathbb{R}^n.
\]
Otherwise stated, the Fourier transform of $(q-\tilde{q})\chi_\Omega$, where $\chi_\Omega$ is the characteristic function of $\Omega$, is identically zero in $\mathscr{S}'(\mathbb{R}^n)$. By the injectivity of the Fourier transformation, this entails that $q=\tilde{q}$ in $\Omega$.

%%%%%%%%%%%%%%%%%%%%%%%%%%%%%%%%%%%%%%%%%%%%%%%%%%%%%%%%
%%%%%%%%%%%                                Proof of Theorem 1.2                                 %%%%%%%%%%%%%%
%%%%%%%%%%%%%%%%%%%%%%%%%%%%%%%%%%%%%%%%%%%%%%%%%%%%%%%%

\subsection{Proof of Theorem \ref{theorem2}}
\label{sec-prthm2}
Pick $\omega$ and $\theta$ be in $\mathbb{S}^{n-1}$, and let $\lambda\in \mathbb{C}\setminus\mathbb{R}$. 
We use the same notations as in the proof of Theorem \ref{theorem1}. Namely,  for all $x \in \Gamma$,
we write
$$
g(x)=g_\lambda(x)=(i\sqrt{\lambda}\omega \cdot \nu+\alpha)e^{i\sqrt{\lambda}\omega \cdot x},\
\overline{h}(x)=\overline{h}_\lambda(x)= (-i\sqrt{\lambda}\theta \cdot \nu+\alpha)e^{-i\sqrt{\lambda}\theta \cdot x}$$
and we recall that
$S(\lambda,\omega ,\theta)=\int_\Gamma u_\lambda (g)\overline{h}ds(x)$. Next, for all $\mu \in \rho(A)\cap \rho(\tilde{A})$ we set
\begin{equation}
\label{4.0}
T(\lambda ,\mu)=T(\lambda ,\mu,\omega ,\theta):=S(\lambda,\omega ,\theta)-S(\mu,\omega ,\theta)=\int_\Gamma \left(u_\lambda (g)-u_\mu (g) \right)\overline{h}ds(x).
\end{equation}
By Lemma \ref{lemma4}, we have
%Since
%\[
%u_\lambda(g){_{|\Gamma}} -u_\mu(g){_{|\Gamma}}=(\lambda-\mu)\sum_{k\ge 1}\frac{(g|\psi_k)}{(\lambda_k-\lambda)(\lambda_k-\mu)}\psi_k,
%\]
%and the series converges in $H^{1/2}(\Gamma)$, according to Lemma \ref{lemma4}, we have
\[
T(\lambda ,\mu)= (\lambda-\mu)\sum_{k\ge 1}\frac{d_k}{(\lambda_k-\lambda)(\lambda_k-\mu)},\ d_k:=(g,\psi_k)(\psi_k,h),
\]
%where $d_k:=(g|\psi_k)(\psi_k|h)$, according to Lemma \ref{lemma4}.
and hence
\begin{equation}
\label{dt}
T(\lambda ,\mu)-\tilde{T}(\lambda ,\mu)=U(\lambda ,\mu)+V(\lambda ,\mu),
\end{equation}
where
\begin{eqnarray}
U(\lambda ,\mu)&:=&\sum_{k\ge 1} \frac{\lambda-\mu}{\lambda_k-\mu} \frac{d_k-\tilde{d}_k}{\lambda_k-\lambda}, \label{du}
\\
V(\lambda ,\mu)&:=&\sum_{k\ge 1}\left(\frac{\lambda-\mu}{(\lambda_k-\lambda)(\lambda_k-\mu)}-\frac{\lambda-\mu}{(\tilde{\lambda}_k-\lambda)(\tilde{\lambda}_k-\mu)}\right)\tilde{d}_k. \label{dv}
\end{eqnarray}

Notice that for all $k \in \mathbb{N}$, we have $d_k-\tilde{d}_k=(g,\psi_k-\tilde{\psi}_k) (\psi_k,h)+ (g,\tilde{\psi}_k)(\psi_k-\tilde{\psi}_k,h)$, which immediately entails that
%\[
%|d_k-\tilde{d}_k|\le |(g|\psi_k)|\|\psi_k-\tilde{\psi}_k\|_{L^2(\Gamma)}\|h\|_{L^2(\Gamma)}+|(\tilde{\psi}_k|h),\psi_k-\tilde{\psi}_k\|_{L^2(\Gamma)}\|g\|_{L^2(\Gamma)}.
%\]
%Hence
\begin{equation}
\label{es0}
\frac{|d_k-\tilde{d}_k|}{|\lambda_k-\lambda|}\le \left(\frac{|(g|\psi_k)|}{|\lambda_k-\lambda|}\|h\|_{L^2(\Gamma)}+\rho_k(\lambda)\frac{|(\tilde{\psi}_k|h)|}{|\tilde{\lambda}_k-\lambda|}\|g\|_{L^2(\Gamma)}\right)\|\psi_k-\tilde{\psi}_k\|_{L^2(\Gamma)},
\end{equation}
where $\rho_k(\lambda):=|\tilde{\lambda}_k-\lambda|/|\lambda_k-\lambda|$. Further, since $0 \le \rho_k(\lambda) \le 1+|\lambda_k-\tilde{\lambda}_k| / |\lambda_k-\lambda|$ and $(\lambda_k-\tilde{\lambda}_k)\in \ell^\infty$ by assumption, with $\|(\lambda_k-\tilde{\lambda}_k)\|_{\ell^\infty}\le \aleph$, it is apparent that
$(\rho_k(\lambda ))\in \ell ^\infty$ and that
$$
\|(\rho_k(\lambda ))\|_{\ell ^\infty}\le \zeta(\lambda):=1+\frac{\aleph}{|\Im \lambda|},
$$
where $\Im \lambda$ denotes the imaginary part of $\lambda$.
Thus, applying the Cauchy-Schwarz inequality in \eqref{es0} and Parseval's theorem to the representation formula \eqref{rep1} in Lemma \ref{lemma2}, we get that
\begin{equation}
\label{es1}
\sum_{k=1}^N \frac{|d_k-\tilde{d}_k|}{|\lambda_k-\lambda|}\le M(\lambda)\|(\psi_k-\tilde{\psi}_k)\|_{\ell^2(L^2(\Gamma))},\ N \in \mathbb{N},
\end{equation}
where
\begin{equation}
\label{def-M}
M(\lambda):=\|h\|_{L^2(\Gamma)}\|u_\lambda (g)\|_H +\zeta(\lambda)\|g\|_{L^2(\Gamma)}\|\tilde{u}_\lambda (h)\|_H.
\end{equation}
As a consequence we have $\sum_{k\ge 1} |d_k-\tilde{d}_k| / |\lambda_k-\lambda|<\infty$. Furthermore, taking into account that
$\mu \in (-\infty ,-(\lambda^\ast +1)]\mapsto (\lambda-\mu)/(\lambda_k-\mu)$ is bounded according to \eqref{lb},
%$$ \frac{|\lambda-\mu|}{|\lambda_k-\mu|} \le 1 + \frac{|\lambda|}{\lambda_1},\ \mu \in (-\infty,-\lambda_1], $$
we apply the dominated convergence theorem to \eqref{du} and find that
\begin{equation}\label{4.1}
\lim_{\mu=\Re \mu \rightarrow -\infty}U(\lambda ,\mu)=\sum_{k\ge 1}\frac{d_k-\tilde{d}_k}{\lambda_k-\lambda}=:\mathcal{U}(\lambda).
\end{equation}
Moreover, we have
\begin{equation}\label{4.4}
|\mathcal{U}(\lambda)| %\le \sum_{k\ge 1}\frac{|d_k-\tilde{d}_k|}{|\lambda_k-\lambda|}
\le M(\lambda)\|(\psi_k-\tilde{\psi}_k)\|_{\ell^2(L^2(\Gamma))},
\end{equation}
according to \eqref{es1}.

Arguing as before with $V$ defined by \eqref{dv} instead of $U$, we obtain in a similar fashion that
\begin{equation}\label{4.3}
\lim_{\mu=\Re \mu \rightarrow -\infty}V(\lambda ,\mu)=\sum_{k\ge 1} \frac{\tilde{\lambda}_k-\lambda_k}{(\lambda_k-\lambda)(\tilde{\lambda}_k-\lambda)} \tilde{d}_k=:\mathcal{V}(\lambda)
%:=\sum_{k\ge 1}\left(\frac{1}{\lambda_k-\lambda}-\frac{1}{\tilde{\lambda}_k-\lambda}\right)\tilde{d}_k
\end{equation}
and that
%\[
%|\mathcal{V}(\lambda)|\le \sum_{k\ge 1} \rho_k(\lambda)|\tilde{\lambda}_k-\lambda_k|\frac{|d_k|}{(\tilde{\lambda}_k-\lambda)^2},
%\]
%from which we derive in a straightforward manner
\begin{equation}\label{4.7}
|\mathcal{V}(\lambda)|\le \zeta(\lambda)\|(\tilde{\lambda}_k-\lambda_k)\|_{\ell^\infty}\| \tilde{u}_\lambda(g)\|_H\|\tilde{u}_\lambda(h)\|_H.
\end{equation}

Having seen this, we refer to \eqref{4.0}-\eqref{dt} and deduce from Lemma \ref{lemma5.1}, \eqref{4.1} and \eqref{4.3} that
\begin{equation}
\label{es3}
\int_\Gamma \left( u_\lambda (g)-\tilde{u}_\lambda (g) \right) \overline{h}ds(x)=\mathcal{U}(\lambda)+\mathcal{V}(\lambda).
\end{equation}
Now, taking $\lambda=\lambda_\tau=(\tau+i)^2$ for some fixed $\tau \in \left(  |\xi|/2, \infty \right)$ and $(\omega,\theta)=(\omega_\tau,\theta_\tau)$, where $\omega_\tau$ and $\theta_\tau$ are the same as in \eqref{es4}, we combine \eqref{S1b}-\eqref{id3} with \eqref{es3}. We obtain that
the Fourier transform $\hat{b}$ of $b:=(\tilde{q}-q)\chi_\Omega$, reads
\begin{equation}\label{4.6}
\hat{b}((1+i/\tau)\xi)= \mathcal{U}(\lambda_\tau )+\mathcal{V}(\lambda_\tau)+\mathfrak{R}(\lambda_\tau),
\end{equation}
where
\[
\mathfrak{R}(\lambda_\tau):=\int_\Omega (\tilde{A}-\lambda_\tau)^{-1}(\tilde{q}\mathfrak{e}_{\lambda_\tau,\omega_\tau})\tilde{q}\mathfrak{e}_{\lambda_\tau,-\theta_\tau}dx-\int_\Omega (A-\lambda_\tau)^{-1}(q\mathfrak{e}_{\lambda_\tau,\omega_\tau})q\mathfrak{e}_{\lambda_\tau,-\theta_\tau}dx.
\]
Moreover, for all $\tau \ge \max(|\xi|/2,\tau_\ast)$, we have
\begin{equation}\label{4.8.1}
|\mathfrak{R}(\lambda_\tau)|\le C \tau^{-1+2\beta},
\end{equation}
by \eqref{gs1}-\eqref{gm0}, where $\beta \in [0,1/2)$ is defined in Theorem \ref{theorem2} and $\tau_\ast$ is the same as in Corollary \ref{corollary1}.
Here and in the remaining part of this proof, $C$ denotes a positive constant depending only on $n$, $\Omega$, $\aleph$ and $\mathfrak{c}$, which may change from line to line.  

On the other hand, using that
\begin{eqnarray*}
\left| \hat{b}((1+i/\tau)\xi) -\hat{b}(\xi) \right| & = & \left| \int_{\mathbb{R}^n} e^{-i \xi \cdot x} \left( e^{\frac{\xi}{\tau} \cdot x} - 1 \right) b(x) dx \right| \\
& \le & \frac{| \xi |}{\tau} \left( \sup_{x \in \Omega} e^{(| \xi | / \tau) |x|} \right) \| b \|_{L^1(\mathbb{R}^n)},
\end{eqnarray*}
%and that $| \xi | / \tau \le 2$, 
we get in a similar way to \cite[Eq. (5.1)]{CS} that
$$
|\hat{b}(\xi)|\le |\hat{b}((1+i/\tau)\xi)|+\frac{c|\xi|}{\tau}e^{c|\xi|/\tau}\aleph,\ \tau \in (|\xi|/2, \infty),
$$
for some positive constant $c$ depending only on $\Omega$. Putting this together with
\eqref{4.6}-\eqref{4.8.1} we find that for all $\tau \ge \max(|\xi|/2,\tau_\ast)$,
\begin{equation}\label{4.9}
|\hat{b}(\xi)|\le \frac{C}{\tau^{1-2\beta}}+\frac{c|\xi|}{\tau}e^{c|\xi|/\tau}\aleph+|\mathcal{U}(\lambda_\tau )|+|\mathcal{V}(\lambda_\tau)|.
\end{equation}

To upper bound $|\mathcal{U}(\lambda_\tau )|+|\mathcal{V}(\lambda_\tau)|$ on the right hand side of \eqref{4.9}, we recall from \eqref{sol1} that $u_{\lambda_\tau}(g)=-(A-\lambda_\tau)^{-1} (q \mathfrak{e}_{\lambda_\tau,\omega_\tau}) + \mathfrak{e}_{\lambda_\tau,\omega_\tau}$ and that $\tilde{u}_{\lambda_\tau}(h)=-(\tilde{A}-\lambda_\tau)^{-1} (\tilde{q} \mathfrak{e}_{\lambda_\tau,-\theta_\tau})+\mathfrak{e}_{\lambda_\tau,-\theta_\tau}$, and we combine \eqref{re7} with
\eqref{gs0} and \eqref{gs1}: We get for all $\tau \ge \tau_{\xi}:=\max(1,|\xi|/2,\tau_\ast)$, that
$$\| u_{\lambda_\tau}(g) \|_H + \| \tilde{u}_{\lambda_\tau}(h) \|_H \leq C. $$
This together with the basic estimate $\| g \|_{L^2(\Gamma)}  + \| h \|_{L^2(\Gamma)} \le C \tau$, \eqref{def-M}, \eqref{4.4} and \eqref{4.7}, yield that
$$
|\mathcal{U}(\lambda_\tau )|+|\mathcal{V}(\lambda_\tau)|\le C\left(\tau\|(\psi_k-\tilde{\psi}_k)\|_{\ell^2(L^2(\Gamma))}+\|(\tilde{\lambda}_k-\lambda_k)\|_{\ell^\infty}\right),\ \tau \in [ \tau_{\xi},\infty ).
$$
Inserting this into \eqref{4.9}, we find that 
\begin{equation}\label{4.12}
|\hat{b}(\xi)|\le \frac{C}{\tau^{1-2\beta}}+\frac{c|\xi|}{\tau}e^{c|\xi|/\tau}\aleph+C \tau \delta,\ \tau \in [ \tau_{\xi},\infty ),
\end{equation}
where we have set
\begin{equation}
\label{def-delta}
\delta:= \| (\psi_k-\tilde{\psi}_k)\|_{\ell^2(L^2(\Gamma))}+\|(\tilde{\lambda}_k-\lambda_k)\|_{\ell^\infty}.
\end{equation}
Let $\varrho \in (0,1)$ to be made precise further. For all $\tau \in [\tau_\ast,\infty)$, where $\tau_\ast$ is defined in Corollary \ref{corollary1}, it is apparent that the condition $\tau \ge \tau_{\xi}$ is automatically satisfied whenever $\xi \in B(0,\tau^\varrho):= \{ \xi \in \mathbb{R}^n,\ |\xi|< \tau^\varrho \}$.
Thus, squaring both sides of \eqref{4.12} and integrating the obtained inequality over $B(0,\tau^\varrho)$, we get that
$$
\|\hat{b}\|_{L^2(B(0,\tau ^\varrho))}^2\le C\left( \tau^{-2(1-2 \beta) + \varrho n}+ e^{2c\tau^{-(1-\varrho)}} \tau^{\varrho (n+2)-2}+\tau^{2+\varrho n}\delta ^2 \right),\ \tau \in [\tau_\ast,\infty).
$$
Then, taking $\varrho=(1-2\beta)/(n+2)$ in the above line, we obtain that
\begin{equation}
\label{sta1}
\|\hat{b}\|_{L^2(B(0,\tau ^{(1-2\beta)/(n+2)}))}^2\le C \left( \tau^{-(1-2\beta)}+\tau^{(3n+4)/(n+2)}\delta ^2 \right), \tau \in [\tau_\ast,\infty).
\end{equation}
On the other hand, using that the Fourier transform is an isometry from $L^2(\mathbb{R}^n)$ to itself, we have for all $\tau \in [\tau_\ast,\infty)$,
\begin{eqnarray*}
\int_{\mathbb{R}^n \setminus B(0,\tau^{(1-2\beta)/(n+2)})} (1+|\xi|^2|)^{-1}|\hat{b}(\xi)|^2d\xi & \le & %\tau^{-2/(n+2)}\|\hat{b}\|_{L^2(\mathbb{R}^n)}^2=
\tau^{-2(1-2\beta)/(n+2)}\|b\|_{L^2(\mathbb{R}^n)}^2 \\
& \le & C \tau^{-2(1-2\beta)/(n+2)},
\end{eqnarray*}
which together with
\eqref{sta1} yields that
\[
\|b\|_{H^{-1}(\mathbb{R}^n)}^2\le \tau^{-2(1-2\beta)/(n+2)}+\tau^{(3n+4)/(n+2)}\delta ^2,\ \tau \in [\tau_\ast,\infty). 
\]
Assuming that $\delta < \left( 2(1-2\beta)/(3n+4) \right)^{1 /2}=:\delta_0$, we get by minimizing the right hand side of the above estimate with respect to $\tau \in [\tau_\ast,\infty)$, that
$$
\|b\|_{H^{-1}(\mathbb{R}^n)}\le C\delta^{2(1-2\beta)/(3(n+2))},
$$
and the desired stability inequality follows from this upon recalling that $\|q-\tilde{q}\|_{H^{-1}(\Omega)}\le \|b\|_{H^{-1}(\mathbb{R}^n)}$. Finally, we complete the proof by noticing that for all $\delta \ge \delta_0$, we have
 $$ \|q-\tilde{q}\|_{H^{-1}(\Omega)}\le \|q-\tilde{q}\|_{L^2(\Omega)} \leq \left( 2 \aleph \delta_0^{-2(1-2\beta)/(3(n+2))} \right) \delta^{2(1-2\beta)/(3(n+2))}. $$

%%%%%%%%%%%%%%%%%%%%%%%%%%%%%%%%%%%%%%%%%%%%%%%%%%%%%%%%
%%%%%%%%%%%                                Proof of Theorem 1.3                                 %%%%%%%%%%%%%%
%%%%%%%%%%%%%%%%%%%%%%%%%%%%%%%%%%%%%%%%%%%%%%%%%%%%%%%%

\subsection{Proof of Theorem \ref{theorem3}}
Upon possibly substituting $q+\lambda^\ast+1$ (resp., $\tilde{q}+\lambda^\ast+1$) for $q$ (resp., $\tilde{q}$), we shall assume without loss of generality in the sequel, that $\lambda_k\ge 1$ (resp., $\tilde{\lambda}_k \ge 1$) for all $k \ge 1$.
%We turn now to proving the statement of Theorem \ref{theorem3}.
Next, taking into account that $q=\tilde{q}$ in $\Omega_0$, we notice that the function $u_k:=\phi_k-\tilde{\phi}_k$, $k \ge 1$, satisfies
\begin{equation}
\label{es20}
(-\Delta +q-\lambda_k)u_k=(\lambda_k-\tilde{\lambda}_k)\tilde{\phi}_k\; \mbox{in}\; \Omega_0,\quad \partial_\nu u_k+\alpha u_k=0\; \mbox{on}\; \Gamma.
\end{equation}
Let $s \in (0,1/2)$ fixed arbitrarily. It follows from \cite[Theorem 1.1]{Ch2022} (with $\eta_1=1/2$ and $\eta_0$ chosen so that $(1/2-\eta_0)/(1+\eta_0)=s/4$) and \cite[comments in Section 1.3]{Ch2022} that 
there exist three  constants
$C=C(n,\Omega_0,\Gamma_{\ast})>0$, $\mathfrak{b}=\mathfrak{b}(n,\Omega_0,\Gamma_{\ast},s)>0$ and $\gamma=\gamma(n,\Omega_0)>0$, such that for all $r \in (0,1)$ and all $\lambda \in [0,+\infty)$, we have
\begin{equation}\label{UC1}
C\left(\|u\|_{H^1(\Gamma_0)}+ \|\partial_\nu u\|_{L^2(\Gamma_0)}\right)\le r^{s/4}\|u\|_{H^2(\Omega_0)}+e^{\mathfrak{b}r^{-\gamma}}\mathfrak{C}_\lambda (u),\ u\in H^2(\Omega_0),
\end{equation}
%whenever $r \in (0,1)$ and all $\lambda \in [0,+\infty)$, 
where we have set $\Gamma_0:=\partial \Omega_0$ and
\[
\mathfrak{C}_\lambda (u) :=(1+\lambda)\left(\|u\|_{H^1(\Gamma_{\ast})}+\| \partial_\nu u\|_{L^2(\Gamma_{\ast})}\right)+\|(\Delta -q+\lambda)u\|_{L^2(\Omega_0)}.
\]
% where $C$ depends only on $n$, $\Omega$ and  $\Gamma_{\ast}$, whereas $\mathfrak{b}$ depends only on $n$, $\Omega$, $\Gamma_{\ast}$ and $s$
Thus, in light of \eqref{5.1} and the embedding $\Gamma \subset \Gamma_0$, we deduce from \eqref{es20} upon applying \eqref{UC1} with $(\lambda,u)=(\lambda_k,(u_k)_{| \Omega_0})$, $k \geq 1$, that for all $r \in (0,1)$, we have
\begin{eqnarray*}
& & C\|\psi_k-\tilde{\psi}_k\|_{L^2(\Gamma)} \\
& \le & r^{s/4}( \lambda_k +\tilde{\lambda}_k )+e^{\mathfrak{b} r^{-\gamma}}\left((1+\|\alpha\|_{C^{0,1}(\Gamma)})\lambda_k\|\psi_k-\tilde{\psi}_k\|_{H^1(\Gamma_{\ast})}+|\lambda_k-\tilde{\lambda}_k|\right),
\end{eqnarray*}
for some constant $C>0$ depending only on  $n$, $\Omega$, $\Omega_0$, $\Gamma_\ast$, $\aleph$ and $s$. 
%Notice that we used that $\Gamma \subset \Gamma_0$.
From this and Weyl's asymptotic formula \eqref{waf}, it then follows for all $k \ge 1$ and all $r \in (0,1)$, that
\begin{equation}
\label{es21}
C\|\psi_k-\tilde{\psi}_k\|_{L^2(\Gamma)}^2\le r^{s/2}k^{4/n}+e^{2\mathfrak{b}r^{-\gamma}}\left(k^{4/n}\|\psi_k-\tilde{\psi}_k\|_{H^1(\Gamma_{\ast})}^2 + |\lambda_k-\tilde{\lambda}_k|^2 \right).
\end{equation}
Here and in the remaining part of this proof, $C$ denotes a generic positive constant depending only on $n$, $\Omega$, $\Omega_0$, $\Gamma_\ast$, $\aleph$ and $\alpha$, which may change from one line to another. Since the constant $C$ is independent of $k \geq 1$ and
since $\sum_{k\ge 1}k^{-2\mathfrak{t}+4/n}<\infty$ as we have $2\mathfrak{t}>1+4/n$,  we find upon multiplying both sides of 
\eqref{es21} by $k^{-2\mathfrak{t}}$ and then summing up the result over $k \geq 1$, that
\begin{eqnarray}
\label{es22}
& & C \|(k^{-\mathfrak{t}} (\psi_k-\tilde{\psi}_k))\|_{\ell ^2(L^2(\Gamma))}^2 \\
%\sum_{k\ge 1} k^{-2\mathfrak{t}}\|\psi_k-\tilde{\psi}_k\|_{L^2(\Gamma)}^2 \\
& \le & r^{s/2}
+e^{2\mathfrak{b}r^{-\gamma}} \left(  \| (k^{-\mathfrak{t}+2/n} (\psi_k-\tilde{\psi}_k))\|_{\ell^2(H^1(\Gamma_{\ast}))}^2 + \| (k^{-\mathfrak{t}} (\lambda_k-\tilde{\lambda}_k)) \|_{\ell^2}^2 \right), \nonumber
% & & \sum_{k \ge 1} k^{-2\mathfrak{t}} \left(|\lambda_k-\tilde{\lambda}_k|^2+k^{4/n}\|\psi_k-\tilde{\psi}_k\|_{H^1(\Gamma_{\ast})}^2\right).
\end{eqnarray}
uniformly in $r \in (0,1)$. 
Further, taking into account that $(k^{\mathfrak{t}}(\psi_k-\tilde{\phi}_k)) \in \ell^2(L^2(\Gamma))$ and $\|(k^{\mathfrak{t}}(\psi_k-\tilde{\phi}_k))\|_{\ell ^2(L^2(\Gamma))}\le \aleph$, we have
\begin{eqnarray*}
\|(\psi_k-\tilde{\psi}_k)\|_{\ell ^2(L^2(\Gamma))}^2
&\le &\|(k^{\mathfrak{t}}(\psi_k-\tilde{\phi}_k))\|_{\ell ^2(L^2(\Gamma))}\|(k^{-\mathfrak{t}}(\psi_k-\tilde{\psi}_k))\|_{\ell ^2(L^2(\Gamma))} \\
& \le & \aleph\|(k^{-\mathfrak{t}}(\psi_k-\tilde{\psi}_k))\|_{\ell ^2(L^2(\Gamma))},
\end{eqnarray*}
by the Cauchy-Schwarz inequality, and hence
\begin{eqnarray}
\label{es23}
& & C \|(\psi_k-\tilde{\psi}_k)\|_{\ell ^2(L^2(\Gamma))}^2 \\
& \le & r^{s/4}
+e^{\mathfrak{b}r^{-\gamma}} \left(  \| (k^{-\mathfrak{t}} (\lambda_k-\tilde{\lambda}_k)) \|_{\ell^2} + \| (k^{-\mathfrak{t}+2/n} (\psi_k-\tilde{\psi}_k))\|_{\ell^2(H^1(\Gamma_{\ast}))} \right), \nonumber
\end{eqnarray}
whenever $r \in (0,1)$, by \eqref{es22}. 
Moreover, since 
$$\| (k^{-\mathfrak{t}} (\lambda_k-\tilde{\lambda}_k)) \|_{\ell^2} \le \left( \sum_{k\ge 1} k^{-2\mathfrak{t}}\right)^{1/2} 
\| (\lambda_k-\tilde{\lambda}_k)) \|_{\ell^\infty}$$ 
and $\sum_{k\ge 1} k^{-2\mathfrak{t}}<\infty$ as we assumed that $2\mathfrak{t}>1+n/2$, \eqref{es23} then provides
\begin{equation}
\label{es24}
\|(\psi_k-\tilde{\psi}_k)\|_{\ell ^2(L^2(\Gamma))}^2 \le C \left( r^{s/4}
+e^{\mathfrak{b}r^{-\gamma}} \delta_\ast \right),\ r \in (0,1),
\end{equation}
where we have set
$$ \delta_\ast := 
\| (\lambda_k-\tilde{\lambda}_k) \|_{\ell^\infty} + \| (k^{-\mathfrak{t}+2/n} (\psi_k-\tilde{\psi}_k))\|_{\ell^2(H^1(\Gamma_{\ast}))}. $$
Next, with reference to \eqref{def-delta} we have
\begin{eqnarray*}
\delta^2 & \le & 2 \left( \|(\lambda_k-\tilde{\lambda}_k)\|_{\ell^\infty}^2 + \|(\psi_k-\tilde{\psi}_k)\|_{\ell ^2(L^2(\Gamma))}^2 \right) \\
& \le & 2 \left( \aleph \|(\lambda_k-\tilde{\lambda}_k)\|_{\ell^\infty} + \|(\psi_k-\tilde{\psi}_k)\|_{\ell ^2(L^2(\Gamma))}^2 \right). \nonumber
\end{eqnarray*}
Moreover, since $\|(\lambda_k-\tilde{\lambda}_k)\|_{\ell^\infty} \le e^{\mathfrak{b}r^{-\gamma}} \delta_\ast$ whenever $r \in (0,1)$, the above inequality combined with \eqref{es24} yield that
\begin{equation}
\label{e1}
\delta^2 \leq C \left( r^{s/4}+e^{\mathfrak{b}r^{-\gamma}} \delta_\ast \right),\  r \in (0,1).
\end{equation}

On the other hand, we have 
$$
\|q-\tilde{q}\|_{H^{-1}(\Omega)}\le C \delta^{2 (1-2\beta) / (3(n+2))},
$$
from Theorem \ref{theorem2}. Putting this together with \eqref{e1}, we obtain that
\begin{equation}\label{e3}
\|q-\tilde{q}\|_{H^{-1}(\Omega)}\le C \left( r^{s/4}+e^{\mathfrak{b}r^{-\gamma}}\delta_\ast \right)^{(1-2\beta)/(3(n+2))},\ r \in (0,1).
\end{equation}
Let us now examine the two cases $\delta_\ast \in (0,1/e)$ and $\delta_\ast \in [1/e,\infty)$ separately. We start with $\delta_\ast \in (0,1/e)$ and take $r=| \ln \delta_\ast |^{-1/\gamma} \in (0,1)$ in \eqref{e3}, getting that
\begin{eqnarray*}
\|q-\tilde{q}\|_{H^{-1}(\Omega)} & \le & C \left( | \ln \delta_\ast |^{-s/(4\gamma)}+\delta_\ast^{(\mathfrak{b}+1)} \right)^{(1-2\beta)/(3(n+2))} \\
& \le & C \left( | \ln \delta_\ast |^{-s/(4\gamma)}+ e^{-(\mathfrak{b}+1)} | \ln \delta_\ast |^{-(\mathfrak{b}+1)} \right)^{(1-2\beta)/(3(n+2))},
\end{eqnarray*}
where we used in the last line that $\delta_\ast \le 1/(e | \ln \delta_\ast |)$. This immediately yields
\begin{equation}\label{e4} 
\|q-\tilde{q}\|_{H^{-1}(\Omega)} \le C | \ln \delta_\ast |^{-\vartheta},\ \delta_\ast \in (0,1/e),
\end{equation}
where
$$\vartheta:=\min \left( s/(4\gamma) , \mathfrak{b}+1 \right)(1-2\beta)/(3(n+2)).$$ 
Next, for $\delta_\ast \in [1/e,\infty)$, we get upon choosing, say, $r=1/2$ in \eqref{e3}, and then taking into account that 
$r <  1 \le e \delta_\ast$ and $(1-2\beta)/(3(n+2)) \ge 0$, that
\begin{eqnarray*}
\|q-\tilde{q}\|_{H^{-1}(\Omega)} & \le & C \left( (e\delta_\ast)^{s/4}+e^{2^\gamma \mathfrak{b}-1} e \delta_\ast \right)^{(1-2\beta)/(3(n+2))} \\
& \le & C (e\delta_\ast)^{(1-2\beta)/(3(n+2))} \\
& \le & C \delta_\ast. % \delta_\ast \in [1/e,\infty).
\end{eqnarray*}
Now, with reference to \eqref{def-Phi}, the stability estimate \eqref{thm3} follows readily from this and \eqref{e4}.

\appendix
\section{Proof of the continuity of $\mathfrak{a}_0$}\label{appendixA}
As $\Gamma$ is compact, $L^{s_2}(\Gamma)$ is continuously embedded in $L^{s_1}(\Gamma)$ whenever $s_1\le s_2$. Therefore, we may assume without loss of generality that $s\in (n-1,n)$. Set $h(r):=r/(r-2)$, $r \in (2,\infty)$. It is easy to see that the function $h$ is decreasing and bijective from $\left( 2n/(n-1),2(n-1)/(n-2) \right)$ onto $(n-1,n)$. 
As a consequence, there exists a unique $p\in \left( 2n/(n-1),2(n-1)/(n-2) \right)$ such that $h(p)=s$, since $s\in (n-1,n)$. Otherwise stated, the conjugate exponent of $p/2$ is $s$, i.e. $s=(p/2)^\ast$. Further, we recall from the comments following \cite[Theorem 1.4.1]{Gr} that the space $H^1(\Omega)=W^{1,2}(\Omega)$ is continuously embedded into $W^{t,p}(\Omega)$, where $t=1-n/2+n/p\in \left( (n-2)/\left( 2(n-1)\right),1/2 \right)$. And since $t-1/p>0$, the map $w\in W^{t,p}(\Omega)\mapsto u_{|\Gamma}\in L^p(\Gamma)$ is bounded according to \cite[Theorem 1.6.1.3]{Gr}. Therefore, $w\in V\mapsto w_{|\Gamma}\in L^p(\Gamma)$ is bounded as well.

Now, for all $u,v\in V$, we have
$$
\int_\Gamma | \alpha | | u | | v | ds(x)\le \|\alpha |u|^2\|_{L^1(\Omega)}^{1/2}\|\alpha |v|^2\|_{L^1(\Omega)}^{1/2},
$$
by applying the Cauchy-Schwarz inequality and then the H\"older inequality, from where we get that
\begin{equation}\label{A1}
\int_\Gamma | \alpha | | u | | v | ds(x)\le \|\alpha\|_{L^s(\Gamma)}\|u\|_{L^p(\Gamma)}\|v\|_{L^p(\Gamma)}.
\end{equation}
Thus, bearing in mind that $|\mathfrak{a}_0(u,v)| \le \int_\Gamma |\alpha| |u||v|ds(x)$, we find that
\[
|\mathfrak{a}_0(u,v)| \le c_0\|\alpha\|_{L^s(\Gamma)}\|u\|_V\|v\|_{V},
\]
where $c_0$ is a positive constant depending only on $\Omega$ and $s$.

%%%%%%%%%%%%%%%%%%%%%%%%%%%%%%%%%%%%%%%%%%%%%%%%%%%%%%%%
%%%%%%%%%%%                                        Bibliography                                        %%%%%%%%%%%%%%
%%%%%%%%%%%%%%%%%%%%%%%%%%%%%%%%%%%%%%%%%%%%%%%%%%%%%%%%

\end{document}